\documentclass[11pt, reqno]{amsart}

\usepackage{enumitem}

\usepackage{geometry}
\usepackage{color}
\usepackage{amssymb}
\usepackage{amsmath}
\usepackage{arydshln}
\usepackage{hyperref}
\usepackage{bbm}
\usepackage{mathrsfs}

\usepackage{enumitem}
\def\BibTeX{{\rm B\kern-.05em{\sc i\kern-.025em b}\kern-.08em
    T\kern-.1667em\lower.7ex\hbox{E}\kern-.125emX}}

\numberwithin{equation}{section}

\newcommand{\R}{\mathbb{R}}

\newcommand{\N}{\mathbb{N}}

\newcommand{\F}{\mathcal{F}}

\renewcommand{\P}{\mathbb{P}}
\newcommand{\E}{\mathbb{E}}
\newcommand{\e}{\varepsilon}

\newcommand{\1}{\mathbbm{1}}

\newtheorem{Theorem}{Theorem}[section]
\newtheorem{Proposition}[Theorem]{Proposition}
\newtheorem{Corollary}[Theorem]{Corollary}
\newtheorem{Lemma}[Theorem]{Lemma}
\newtheorem{Remark}[Theorem]{Remark}
\newtheorem{Definition}[Theorem]{Definition}
\newtheorem{Example}[Theorem]{Example}

\begin{document}

\title[On a class of SPDEs with multiple invariant measures]{On a class of stochastic partial differential equations with multiple invariant measures}

\author{B\'alint Farkas}
\address[B\'alint Farkas]{School of Mathematics and Natural Sciences\\ University of Wuppertal, Germany}
\email[B\'alint Farkas]{farkas@math.uni-wuppertal.de}

\author{Martin Friesen}
\address[Martin Friesen]{School of Mathematics and Natural Sciences\\ University of Wuppertal, Germany}
\email[Martin Friesen]{friesen@math.uni-wuppertal.de}

\author{Barbara R\"udiger}
\address[Barbara R\"udiger]{School of Mathematics and Natural Sciences\\ University of Wuppertal, Germany}
\email[Barbara R\"udiger]{ruediger@uni-wuppertal.de}

\author{Dennis Schroers}
\address[Dennis Schroers]{Department of Mathematics, University of Oslo}
\email[Dennis Schroers]{dennissc@math.uio.no}

\date{\today}



\begin{abstract}
 In this work we investigate the long-time behavior, that is the existence and characterization of invariant measures as well as convergence of transition probabilities, for Markov processes obtained as the unique mild solution to stochastic partial differential equations in a Hilbert space. Contrary to the existing literature where typically uniqueness of invariant measures is studied, we focus on the case where uniqueness of invariant measures fails to hold. Namely, using a \textit{generalized dissipativity condition} combined with a decomposition of the Hilbert space, we prove the existence of multiple limiting distributions in dependence of the initial state of the process, and study convergence of transition probabilities in the Wasserstein 2-distance. Finally, we show that these results contain L\'evy driven Ornstein-Uhlenbeck processes, the Heath-Jarrow-Morton-Musiela equation as well as stochastic partial differential equations with delay as a particular case.
\end{abstract}

\maketitle

\allowdisplaybreaks

\section{Introduction}

Stochastic partial differential equations arise in the modelling of applications in mathematical physics (e.g.{} Navier-Stokes equations \cite{MR2259251, MR1300150, MR2005200, MR2269220} or stochastic non-linear Schr\"odinger equations \cite{MR3474409, MR1954077}),
biology (e.g. catalytic branching processes \cite{MR1921744, MR1653845}), and finance (e.g. forward prices \cite{HJMM92, Tehranchi2005, filipovic2009term}). While the construction of solutions to the underlying stochastic equations is an important mathematical issue, having applications in mind it is indispensable to also study their specific properties. Among them, an investigation of the long-time behavior of solutions, that is existence and uniqueness of invariant measures and convergence of transition probabilities, are often important and at the same time also challenging mathematical topics.
In this work we investigate the long-time behavior of mild solutions to the stochastic partial differential equation of the form
\begin{equation}\label{SPDE}
dX_t=(AX_t+F(X_t)) dt + \sigma(X_t)dW_t + \int_E\gamma(X_{t},\nu)\widetilde{N}(dt,d\nu), \qquad t \geq 0
\end{equation}
on a separable Hilbert space $H$, where $(A,D(A))$ is the generator of a strongly continuous semigroup $(S(t))_{t \geq 0}$ on $H$, $(W_t)_{t \geq 0}$ is a $Q$-Wiener process and $\widetilde{N}(dt,d\nu)$ denotes a compensated Poisson random measure. The precise conditions need to be imposed on these objects will be formulated in the subsequent sections. 

In the literature the study on the existence and uniqueness of invariant measures often relies on different variants of a \textit{dissipativity condition}. The simplest form of such a dissipativity condition is: 
There exists $\alpha > 0$ such that
\begin{align}\label{eq: dissipativity condition}
 \langle Ax - Ay,x-y\rangle_H + \langle F(x) - F(y), x-y \rangle_H \leq - \alpha \| x-y \|_H^2, \qquad x,y \in D(A).
\end{align}
Indeed, if \eqref{eq: dissipativity condition} is satisfied, $\sigma$ and $\gamma$ are globally Lipschitz-continuous, and $\alpha$ is large enough, then there exists a unique invariant measure for the Markov process obtained from \eqref{SPDE}, see, e.g., \cite[Section 16]{peszat2007stochastic}, \cite[Chapter 11, Section 6]{da2014stochastic}, and \cite{Rusinek10}. Note that \eqref{eq: dissipativity condition} is satisfied, if $F$ is globally Lipschitz continuous and $(A,D(A))$ satisfies for some $\beta > 0$ large enough the inequality $\langle Ax,x\rangle_H \leq - \beta \|x\|_H^2$, $x \in D(A)$, i.e. $(A,D(A))$ is the generator of a strongly continuous semigroup satisfying $\| S(t) \|_{L(H)} \leq \mathrm{e}^{- \beta t}$.
Here and below we denote by $L(H)$ the space of bounded linear operators from $H$ to $H$ and by $\| \cdot \|_{L(H)}$ its operator norm.
For weaker variants of the dissipativity condition
(e.g. cases where \eqref{eq: dissipativity condition} only holds for $\|x\|_H, \| y\|_H \geq R$ for some $R > 0$), in general one can neither guarantee the existence nor uniqueness of an invariant measure. Hence, to treat such cases, additional arguments, e.g. coupling methods, are required. Such arguments have been applied to different stochastic partial differential equations on Hilbert spaces in \cite{MR2860444, MR2836539, MR2985090} where existence and, in particular, uniqueness of invariant measures was studied. 
We also mention \cite{MR3178490, MR2857021} for an extension of Harris-type theorems for Wasserstein distances, and \cite{MR3800835, MR2773030} for extensions of coupling methods. 

In contrast to the aforementioned methods and applications, several stochastic models exhibit phase transition phenomena where uniqueness of invariant measures fails to hold. For instance, the generator $(A,D(A))$ and drift $F$ appearing in the Heath-Jarrow-Morton-Musiela equation do not satisfy \eqref{eq: dissipativity condition}, but instead $F$ is globally Lipschitz continuous and the semigroup generated by $(A,D(A))$ satisfies $\| S(t)x - Px\|_H \leq \mathrm{e}^{- \alpha t}\| x - Px\|_H$ for some projection operator $P$. Based on this property it was shown in \cite{Tehranchi2005, Rusinek10} that
the Heath-Jarrow-Morton-Musiela equation has infinitely many invariant measures parametrized by the initial state of the process, see also Section 6. Another example is related to stochastic Volterra equations as studied, e.g., in \cite{BDK19}. There, using a representation of stochastic Volterra equations via SPDEs and combined with some arguments originated from the study of the Heath-Jarrow-Morton-Musiela equation, the authors studied existence of limiting distributions allowing, in particular, that these distributions depend on the initial state of the process. 

In this work we provide a general and unified approach for the study of multiple invariant measures and, moreover, we show that with dependence on the initial distribution the law of the mild solution of \eqref{SPDE} is governed in the limit $t \to \infty$ by one of the invariant measures. In particular, we show that the methods developed in \cite{Tehranchi2005, Rusinek10, BDK19} can be embedded as a special case of a general framework where one replaces \eqref{eq: dissipativity condition} by a weaker dissipativity condition, which we call \textit{generalized dissipativity condition}:
\begin{enumerate}
    \item[(GDC)] There exists a projection operator $P_1$ on the Hilbert space $H$ and there exist constants $\alpha > 0, \beta \geq 0$ such that, for $x,y \in D(A)$, one has:
    \begin{align*}
     &\ \langle Ax - Ay, x - y \rangle_H + \langle F(x) - F(y), x-y \rangle_H \\ &\hskip30mm \leq - \alpha \| x - y \|_H^2 + \left( \alpha + \beta\right) \| P_1 x - P_1 y \|_H^2.
    \end{align*}
\end{enumerate}
Restricting $x,y \in D(A)$ to the subspace $x,y \in \mathrm{ker}(P_1) \cap D(A)$ shows that (GDC) contains the classical dissipativity condition on the subspace $\mathrm{ker}(P_1) \subset H$. Contrary, restricting $x,y$ to $\mathrm{ran}(P_1) \cap D(A)$ does not yield the classical dissipativity condition but instead contains the additional term $\|P_1x - P_1y\|_H^2$ which describes the influence of the non-dissipative part of the drift. Sufficient conditions and additional remarks, e.g., on this condition are collected in the end of Section 2 while particular examples are discussed in Sections 5 -- 7. Let us mention that (GDC) is satisfied if $F$ is globally Lipschitz continuousmthe semigroup $(S(t))_{t \geq 0}$ is symmetric (i.e. $S(t)^* = S(t)$ for $t \geq 0$), and uniformly convergent to $P_1$ (i.e. $\| S(t) - P_1 \|_{L(H)} \longrightarrow 0$ as $t \to \infty$).

Roughly speaking, we will show that under condition (GDC) and some restrictions on the projected coefficients $P_1F$, $P_1\sigma$, and $P_1 \gamma$, the Markov process obtained from \eqref{SPDE} has for each initial data $X_0 = x$ a limiting distribution $\pi_x$ depending only on $P_1x$. Moreover, the transition probabilities converge exponentially fast in the Wasserstein 2-distance to this limiting distributions. In order to prove this result, we first decompose the state space $H$ according to 
\[
 H = H_0 \oplus H_1, \qquad x = P_0x + P_1x, \ \ P_0 := I - P_1,
\]
where $I$ denotes the identity operator on $H$,
and then investigate the components $P_0X_t$ and $P_1X_t$ separately.
Based on an idea taken from \cite{Gaans2005InvariantMF},
we construct, for each $\tau \geq 0$, a coupling of $X_t$ and $X_{t+\tau}$.
This coupling will be then used to efficiently estimate
the Wasserstein 2-distance for the solution started at two different points.

This work is organized as follows. 
In Section 2 we state the precise conditions imposed on the coefficients of the SPDE \eqref{SPDE}, discuss some properties of the solution and then provide sufficient conditions for the generalized dissipativity condition (GDC). Based on condition (GDC) we derive in Section 3 an estimate on the trajectories of the process when started at two different initial points, i.e. we estimate the $L^2$-norm of $X_t^x - X_t^y$ 
when $x \neq y$. Based on this estimate, we state and prove our main results in Section 4. Examples are then discussed in the subsequent Sections 5 -- 7. In Section 5 we explicitly characterize the limiting distributions of the L\'evy driven Ornstein-Uhlenbeck process with an operator $(A,D(A))$ generating an uniformly convergent semigroup.
The Heath-Jarrow-Morton-Musiela equation is then considered in Section 6 for which we first show that the main results of Section 4 contain \cite{Tehranchi2005, Rusinek10}, and then extend these results by 
characterizing its limiting distributions more explicitly. Finally, we apply our results in Section 7 to an SPDE with delay.

\section{Preliminaries}

\subsection{Framework and notation}
Here and throughout this work, $(\Omega,\mathcal{F},(\mathcal{F}_t)_{t\in\mathbb{R}_+},\mathbb{P})$ is a filtered probability space satisfying the usual conditions. 
Let $U$ be a separable Hilbert space and $W=(W_t)_{t\in\mathbb{R}_+}$ be a $Q$-Wiener process with respect to $(\mathcal{F}_t)_{t\in\mathbb{R}_+}$ on $(\Omega,\mathcal{F},(\mathcal{F}_t)_{t\in\mathbb{R}_+},\mathbb{P})$, where $Q:U\to U$ is a non-negative, symmetric, trace class operator.
Let $E$ be a Polish space, $\mathcal{E}$ the Borel-$\sigma$-field on $E$, and $\mu$ a $\sigma$-finite measure on $(E,\mathcal{E})$. Let $N(dt,d\nu)$ be a $(\mathcal{F}_t)_{t \geq 0}$-Poisson random measure with compensator $dt \mu(d\nu)$ and denote by $\widetilde{N}(dt,d\nu) = N(dt,d\nu) - dt \mu(d\nu)$ the corresponding compensated Poisson random measure.
Suppose that the random objects $(W_t)_{t \geq 0}$ and $N(dt,d\nu)$ are mutually independent.

In this work we investigate the long-time behavior of mild solutions to the stochastic partial differential equation
\begin{equation}\label{SPDE1}
\begin{cases}
dX^x_t=(AX^x_t+F(X^x_t)) dt + \sigma(X^x_t)dW_t + \int_E\gamma(X^x_{t},\nu)\widetilde{N}(dt,d\nu), \qquad t \geq 0,\\
X^x_0=x \in L^2(\Omega,\mathcal{F}_0,\mathbb{P};H),
\end{cases}
\end{equation}
where $(A,D(A))$ is the generator of a strongly continuous semigroup $(S(t))_{t \geq 0}$ on $H$, $H \ni x \mapsto F(x) \in H$ and $H \ni x\mapsto \sigma(x) \in L_2^0$ are Borel measurable mappings, and $(x,\nu)\mapsto\gamma(x,\nu)$ is measurable from $( H\times E,\mathcal{B}(H)\otimes\mathcal{E})$ to $(H,\mathcal{B}(H))$.
Here $\mathcal{B}(H)$ denotes the Borel-$\sigma$-algebra on $H$, and $L_2^0:=L_2^0(H)$ is the Hilbert space of all Hilbert-Schmidt operators from $U_0$ to $H$, where $U_0 := Q^{1/2}U$ is a separable Hilbert space endowed with the scalar product 
\[
\langle x,y \rangle_0:=\langle Q^{-1/2} x, Q^{-1/2} y\rangle_U=\sum_{k\in \mathbb{N}}\tfrac{1}{\lambda_k}\langle x,e_k\rangle_U\langle e_k,y\rangle_U,\qquad \forall x,y \in U_0,
\]
and $Q^{-1/2}$ denotes the pseudoinverse of $Q^{1/2}$. Here $(e_j)_{j\in\mathbb{N}}$ denotes an orthogonal basis of eigenvectors of $Q$ in $U$ with corresponding eigenvalues $(\lambda_j)_{j\in\mathbb{N}}$. For comprehensive introductions to integration concepts in infinite dimensional settings we refer to \cite{da2014stochastic} for the case of $Q$-Wiener processes and to \cite{9783319128535} for compensated Poisson random measures as integrators. Throughout this work we suppose that the coefficients $F,\sigma, \gamma$ are Lipschitz continuous. More precisely:
\begin{enumerate}\label{Assumption A1}
    \item[(A1)] There exist constants $L_F, L_{\sigma}, L_{\gamma} \geq 0$ such that for all $x,y \in H$
    \begin{align}
     \| F(x) - F(y) \|_H^2 &\leq L_{F}\| x-y\|_{H}^2, \label{A1 F Lipschtiz}
  \\ \notag \| \sigma(x) - \sigma(y)\|_{L_2^0(H)}^2 &\leq L_{\sigma}\| x-y\|_H^2,
 \\ \int_{E} \| \gamma(x,\nu) - \gamma(y,\nu)\|_H^2 \mu(d\nu) &\leq L_{\gamma} \| x-y\|_H^2. \notag
\end{align}
Moreover we suppose that
\begin{align}\label{eq: moment condition nu}
\int_E \| \gamma(0,\nu)\|_H^2 \mu(d\nu)<\infty.
\end{align}
\end{enumerate}
Note that condition \eqref{eq: moment condition nu} implies that the jumps satisfy the usual growth conditions, i.e.
\begin{align*}
 \int_E \| \gamma(x, \nu)\|_H^2 \mu(d\nu) 
 &\leq 2\int_E \| \gamma(x,\nu) - \gamma(0,\nu)\|_H^2 \mu(d\nu) + 2 \int_{E} \| \gamma(0,\nu)\|_H^2 \mu(d\nu)
 \\ &\leq 2 \max \left\{ L_{\gamma}, \int_{E} \| \gamma(0,\nu)\|_H^2 \mu(d\nu) \right\} (1 + \| x \|_H^2).
\end{align*}
Moreover, it follows from (GDC) and (A1) that 
\[
 \langle Ax, x \rangle_H \leq \left( \beta + \sqrt{L_F}\right) \| x \|_H^2, \qquad x \in D(A).
\]
Hence $A - (\beta + \sqrt{L_F})$ is dissipative and thus by the Lumer-Phillips theorem the semigroup $(S(t))_{t \geq 0}$ generated by $(A,D(A))$ is quasi-contractive, i.e.
\begin{align}\label{eq: semigroup quasi contractive}
 \| S(t)x \|_H \leq \mathrm{e}^{\left( \beta + \sqrt{L_F}\right)t}\|x\|_H , \quad x \in H.
\end{align}
Then, under conditions (GDC) and (A1), for each initial condition  $x\in L^2(\Omega,\mathcal{F}_0,\mathbb{P};H)$ there exists a unique c\'{a}dl\'{a}g, $(\F_t)_{t \geq 0}$-adapted, mean square continuous, mild solution $(X^x_t)_{t\geq 0}$ to \eqref{SPDE1} such that, for each $T > 0$, there exists a constant $C(T) > 0$ satisfying
\begin{align}\label{eq: moments mild solution}
 \mathbb{E} \left[\sup_{t\in [0,T]} \| X^x_t \|_H^2 \right] \leq C(T)\left(1 + \E\left[\|x\|_H^2 \right] \right)
\end{align}
and, for all $x,y \in L^2(\Omega, \F_0, \P; H)$,
\begin{align}\label{eq: continuous dependence on initial condition}
 \mathbb{E} \left[ \| X_t^x - X_t^y \|_H^2 \right] \leq C(T)\E\left[ \| x-y \|_H^2 \right], \qquad t \in [0,T].
\end{align}
This means that $(X_t^x)_{t \geq 0}$ satisfies $\P$-a.s.
\begin{align}\label{SPDE mild formulation}
    X_t^x &= S(t)x + \int_0^t S(t-s)F(X_s^x)ds + \int_0^t S(t-s)\sigma(X_s^x)dW_s
    \\ \notag &\ \ \ + \int_0^t \int_E S(t-s) \gamma(X_s^x, \nu) \widetilde{N}(ds,d\nu), \qquad t \geq 0,
\end{align}
where all (stochastic) integrals are well-defined,
see, e.g., \cite{ALBEVERIO2009835}, \cite{9783319128535}, and \cite{Filipovi2010jump}.
The obtained solution is a Markov process whose transition probabilities $p_t(x,dy) = \P[ X_t^x \in dy]$ 
are measurable with respect to $x$. Denote by $(p_t)_{t \geq 0}$ its transition semigroup, i.e., for each bounded measurable function $f: H \longrightarrow \R$, $p_tf$ is given by
\[
 p_tf(x) = \E\left[ f(X_t^x) \right] = \int_{H}f(y)p_t(x,dy), \qquad t \geq 0, \ \ x \in H.
\]
Using the continuous dependence on the initial condition, see \eqref{eq: continuous dependence on initial condition}, 
it can be shown that $p_tf \in C_b(H)$ for each $f \in C_b(H)$,
i.e. the transition semigroup is $C_b$-Feller.

In this work we investigate the the existence of invariant measures and convergence of the transition probabilities towards these measures for the Markov process $(X_t^x)_{t \geq 0}$ with particular focus on the cases where uniqueness of invariant measures fails to hold.
By slight abuse of notation, we denote by $p_t^*$ the adjoint operator to $p_t$ defined by
\[
 p_t^* \rho(dx) = \int_H p_t(y,dx) \rho(dy), \qquad t \geq 0.
\]
Recall that a probability measure $\pi$ on $(H, \mathcal{B}(H))$ is called \textit{invariant measure} for the semigroup $(p_t)_{t \geq 0}$ if and only if $p_t^* \pi = \pi$ holds for each $t \geq 0$.
Let $\mathcal{P}_2(H)$ be the space of Borel probability measures $\rho$ on $(H, \mathcal{B}(H))$ with finite second moments. Recall that $\mathcal{P}_2(H)$ is separable and complete when equipped with the \textit{Wasserstein-2-distance}
\begin{align}\label{eq: Wasserstein distance}
 \mathrm{W}_2(\rho, \widetilde{\rho}) = \inf_{ G \in \mathcal{H}(\rho, \widetilde{\rho})} \left( \int_{H^2} \| x - y \|_H^2 G(dx,dy) \right)^{\frac{1}{2}}, \qquad \rho, \widetilde{\rho} \in \mathcal{P}_2(H),
\end{align}
where $\mathcal{H}(\rho,\widetilde{\rho})$ denotes the set of all couplings of $(\rho, \widetilde{\rho})$, i.e. Borel probability measures on $H \times H$ whose marginals are given by $\rho$ and $\widetilde{\rho}$, respectively, see \cite[Section 6]{MR2459454} for a general introduction to couplings and Wasserstein distances.

\subsection{Discussion of generalized dissipativity condition}

In this section we briefly discuss the condition
\begin{align}\label{eq: generalized dissipativity for A}
 \langle Ax, x \rangle_H \leq - \lambda_0 \| x\|_H^2 + (\lambda_0 +\lambda_1) \| P_1 x \|_H^2, \qquad x \in D(A),
\end{align}
where $\lambda_0 > 0$ and $\lambda_1 \geq 0$. Note that, if \eqref{eq: generalized dissipativity for A} and condition \eqref{Assumption A1} are satisfied, then 
\begin{align}\label{eq:03}
 & \langle Ax - Ay, x-y\rangle_H + \langle F(x) - F(y), x-y\rangle_H
 \\ \notag &\leq \langle Ax - Ay, x-y\rangle_H + \sqrt{L_F}\| x-y\|_H^2
 \\ \notag &\leq - \left( \lambda_0 - \sqrt{L}_F\right)\| x - y \|_H^2 + \left( \lambda_0 + \lambda_1 \right) \| P_1 x - P_1y \|_H^2,
\end{align}
i.e. the generalized dissipativity condition (GDC) is satisfied for $\alpha = \lambda_0 - \sqrt{L_F}$ and $\beta = \lambda_1 + \sqrt{L_F}$, provided that $\lambda_0 > \sqrt{L_F}$.
\begin{Proposition}\label{prop: orthogonal decomposition of semigroup}
 Suppose that there exists an orthogonal decomposition $H = H_0 \oplus H_1$
 of $H$ into closed linear subspaces $H_0, H_1 \subset H$ such that $(S(t))_{t \geq 0}$ leaves $H_0$ and $H_1$ invariant and there exist constants $\lambda_0 > 0$ and $\lambda_1 \geq 0$ satisfying
   \[
     \| S(t) x_0\|_H\leq \mathrm{e}^{-\lambda_0 t}\| x_0\|_H,
     \qquad \| S(t)x_1 \|_H \leq \mathrm{e}^{\lambda_1 t} \|x_1\|_H,
     \qquad \forall t\geq 0.
    \]
     for all $x_0 \in H_0$ and $x_1 \in H_1$. Then \eqref{eq: generalized dissipativity for A} holds for $P_1$ being the orthogonal projection operator onto $H_1$.
\end{Proposition}
\begin{proof}
 Let $P_0$ be the orthogonal projection operator onto $H_0$.
 Since $(S(t))_{t \geq 0}$ leaves the closed subspace $H_0$ invariant, 
 its restriction $(S(t)|_{H_0})_{t\geq 0}$ onto $H_0$ is a strongly continuous semigroup of contractions on $H_0$ with generator $(A_0,D(A_0))$ being the $H_0$ part of $A$, that is
 \[
  A_0x = Ax, \qquad x \in D(A_0) = \{ y \in D(A) \cap H_0 \ | \ Ay \in H_0 \}. 
 \]
 Since $H_0$ is closed and $S(t)$ leaves $H_0$ invariant, it follows that $Ay = \lim_{t \to 0} \frac{S(t)y - y}{t} \in H_0$ for $y \in D(A) \cap H_0$, i.e. $D(A_0) = D(A) \cap H_0$ and $P_0: D(A) \to D(A_0)$. 
 Let $P_1$ the orthogonal projection operator onto $H_1$. 
 Arguing exactly in the same way shows that
 the restriction $(S(t)|_{H_1})_{t\geq 0}$ is a strongly continuous semigroup of contractions on $H_1$ with generator $(A_1,D(A_1))$
 given by $A_1x = Ax$ and $x \in D(A_1) = D(A) \cap H_1$ so that $P_1: D(A) \to D(A_1)$. Since
 $S(t)$ leaves $H_0$ and $H_1$ invariant, we obtain 
 $P_0S(t) = S(t)P_0$, $P_1 S(t) = S(t)P_1$ from which we conclude that
 $AP_1 x = P_1 Ax$ and $AP_0x = P_0Ax$ for $x \in D(A)$.
 
 
 Since $(\mathrm{e}^{\lambda_0 t}S(t)|_{H_0})_{t\geq 0}$ is a strongly continuous semigroup of contractions on $H_0$ with generator $A_0 + \lambda_0 I$, and $(\mathrm{e}^{-\lambda_1 t}S(t)|_{H_1})_{t\geq 0}$ is a strongly continuous semigroup of contractions on $H_1$ with generator $A_1 - \lambda_1 I$, we have by the Lumer-Phillips theorem (see \cite[Theorem 4.3]{MR710486})
 \[
  \langle A_0x_0, x_0\rangle_H \leq - \lambda_0 \| x_0 \|_H^2 \ \text{ and } \
  \langle A_1x_1, x_1 \rangle_H \leq \lambda_1 \| x_1 \|_H^2,
  \qquad x_0 \in H_0, \ \ x_1 \in H_1.
 \]
 Hence we find that
\begin{align*}
 \langle A x, x \rangle_H 
 &= \langle Ax, P_0 x \rangle_H + \langle Ax, P_1x \rangle_H
 \\ &= \langle P_0 Ax, P_0x \rangle_H + \langle P_1Ax, P_1x \rangle_H
 \\ &=\langle A_0P_0 x , P_0 x \rangle_H + \langle A_1P_1 x,P_1 x\rangle_H 
 \\ &\leq -\lambda_0 \|P_0 x\|_H^2 + \lambda_1 \| P_1 x\|_H^2
 \\ &= - \lambda_0 \| x \|_H^2 + (\lambda_0 + \lambda_1)\| P_1 x\|_H^2,
\end{align*}
where the last equality follows from $H_0 \perp H_1$. This proves the assertion.
\end{proof}
At this point it is worthwhile to mention that
Onno van Gaans has investigated in \cite{Gaans2005InvariantMF} ergodicity for a class of L\'evy driven stochastic partial differential equations where the semigroup $(S(t))_{t \geq 0}$ was supposed to be hyperbolic. Proposition \ref{prop: orthogonal decomposition of semigroup} also covers this case, provided that the hyperbolic decomposition is orthogonal. 
The conditions of previous proposition are satisfied whenever $(S(t))_{t \geq 0}$ is a symmetric, uniformly convergent semigroup.
\begin{Remark} \label{example: ergodic semigroup}
 Suppose that $(S(t))_{t \geq 0}$ is a strongly continuous semigroup on $H$ and there exists an orthogonal projection operator $P$ on $H$ and $\lambda_0 > 0$ such that
 \begin{align}\label{eq: ergodic 00}
  \| S(t)x - Px \|_{H} \leq \mathrm{e}^{- \lambda_0 t}\| x - Px\|_H, \qquad t \geq 0, \ \ x \in H.
 \end{align}
 Then the conditions of Proposition \ref{prop: orthogonal decomposition of semigroup} are satisfied for $H_0 = \mathrm{ker}(P)$ and $H_1 = \mathrm{ran}(P)$ with $\lambda_0 > 0$ and $\lambda_1 = 0$. In particular, $(S(t))_{t \geq 0}$ is a semigroup of contractions.
\end{Remark} 
The following example shows that \eqref{eq: generalized dissipativity for A} can also be satisfied for non-symmetric and non-convergent semigroups.
\begin{Example}
 Let $H = \R^2$, $H_0 = \R \times \{0\}$, $H_1 = \{0\} \times \R$,
 and denote by $P_0, P_1$ the projection operators onto $H_0$ and $H_1$, respectively. Let $A$ be given by $A = \begin{pmatrix} -1 & 1 \\ 0 & 1 \end{pmatrix}$. Then 
 \begin{align*}
  \left\langle \begin{pmatrix} x \\ y \end{pmatrix}, A \begin{pmatrix}x \\ y \end{pmatrix} \right \rangle
  &= -x^2 + xy + y^2
  \\ &\leq - \frac{1}{2} (x^2 + y^2) + 2 y^2
  \\ &= - \frac{1}{2} \| (x,y) \|_H^2 + 2 \| P_1(x,y) \|_H^2,
 \end{align*}
 i.e. \eqref{eq: generalized dissipativity for A} holds for $\lambda_0 = \frac{1}{2}$ and $\lambda_1 = \frac{3}{2}$. Since
 $ \mathrm{e}^{tA} = \begin{pmatrix} \mathrm{e}^{-t} & \frac{\mathrm{e}^t - \mathrm{e}^{-t}}{2} \\ 0 & \mathrm{e}^{t} \end{pmatrix}$,
 it is clear that neither the conditions of Proposition \ref{prop: orthogonal decomposition of semigroup} nor of Remark \ref{example: ergodic semigroup}
 are satisfied.
\end{Example}

\section{Key stability estimate}

Define, for $x,y \in D(A)$, the function
\begin{align*}
\mathcal{L}(\|\cdot\|_H^2)(x,y)
 &:= 2 \langle A (x-y) +  F(x)-F(y), x-y\rangle_H+  \|\sigma(x)-\sigma(y)\|_{L_2^0(H)}^2
 \\ &\ \ \ + \int_E \|\gamma(x,\nu)-\gamma(y,\nu)\|_H^2 \mu(d\nu).\notag
\end{align*}
Remark that under the additional assumption that \eqref{SPDE} has a strong solution, the function 
\begin{equation*}
  \mathcal{L}(\|\cdot\|_H^2)(z):=  2 \langle A (z) +  F(z), z\rangle_H+  \|\sigma(z)\|_{L_2^0(H)}^2+ \int_E \|\gamma(z,\nu)\|_H^2 \mu(d\nu).
\end{equation*} 
is simply the generator $\mathcal{L}$ applied to the unbounded function $\|z\|_H^2$, see, e.g,. \cite[equation (3.4)]{Albeverio2017}). 
Since we work with mild solutions instead, all computations given below 
require to use additionally Yosida approximations for the mild solution of \eqref{SPDE}.

Below we first prove a Lyapunov-type estimate for $\mathcal{L}(\| \cdot \|_H^2)$ and then deduce from that by an application of the It\^{o}-formula for mild solutions to \eqref{SPDE1} an estimate for the $L^2$-norm of $X_t^x - X_t^y$.
\begin{Lemma}\label{L: Lyapunovfunction}
Assume that condition (GDC) and (A1) are satisfied. Then
\begin{equation}\label{LyapunovFunct}
\mathcal{L}(\| \cdot\|_H^2)(x,y)\leq - \left( 2 \alpha - L_{\sigma} - L_{\gamma} \right)  \| x-y\|_H^2+2(\alpha+\beta) \|P_1x- P_1y\|_H^2
\end{equation}
 holds for $x,y \in D(A)$.
\end{Lemma}
\begin{proof}
Using first (A1) and then (GDC) we find that
\begin{align*}
    \mathcal{L}(\| \cdot \|_H^2)(x,y) &\leq (L_{\sigma} + L_{\gamma}) \| x -y \|_H^2 
    \\ &\ \ \ + 2\langle Ax-Ay,x-y \rangle_H + 2\langle F(x) - F(y), x-y \rangle_H
    \\ &\leq - \left( 2\alpha - L_{\sigma} - L_{\gamma} \right)\| x-y\|_H^2 + 2\left(\alpha + \beta\right) \| P_1x - P_1y \|_H^2.
\end{align*}
This proves the asserted inequality.
\end{proof}
The following is our key stability estimate.
\begin{Proposition}\label{T: Lipschitz continuity}
Suppose that (GDC) and (A1) are satisfied, that 
\begin{align}\label{definition epsilon}
 \e := 2\alpha - L_{\sigma} - L_{\gamma} > 0,
\end{align} 
and suppose that
\begin{equation}\label{abstractgrowthsforjumpcoeff}
 \sup_{x\in H}\int_E \|\gamma(x,\nu)\|^4 \mu(d\nu)<\infty.
\end{equation}

Then, for each $x,y\in L^2(\Omega,\mathcal{F}_0,\mathbb{P};H)$ and all $t \geq 0$,
\begin{align}\label{Lipschitzcontinuity}
  & \mathbb{E}\left[\|X_t^x- X_t^y\|_H^2 \right]\notag\\
\leq & \mathrm{e}^{- \e t}\mathbb{E} \left[\| x-y\|_{H}^2 \right]  + 2(\alpha+\beta)  \int_0^t \mathrm{e}^{ - \e (t-s)}\mathbb{E}\left[\| P_1 X_s^x-P_1X_s^y\|_H^2\right]ds.
\end{align}
\end{Proposition}
\begin{proof}
To simplify the notation, denote by $(X_t)_{t\geq 0}$ the mild solution to (\ref{SPDE}) with initial condition $x$ and by $(Y_t)_{t\geq 0}$ the mild solution to (\ref{SPDE}) with initial condition $y$. Moreover, we write $(X_t^n)_{t\geq 0}$ and $(Y_t^n)_{t\geq 0}$ for the strong solutions to the corresponding Yosida -approximation systems
\begin{align*}
 \begin{cases}
dX^n_t=AX^n_t+R_n F(X^n_t) dt +R_n \sigma(X^n_t)dW_t + \int_ER_n\gamma(X^n_{t},\nu)\widetilde{N}(dt,d\nu),\\
X^n_0=R_nx,\quad t \geq 0
\end{cases}
\end{align*}
and 
\begin{align*}
 \begin{cases}
dY^n_t=AY^n_t+R_n F(Y^n_t) dt +R_n \sigma(Y^n_t)dW_t + \int_ER_n\gamma(Y^n_{t},\nu)\widetilde{N}(dt,d\nu),\\
Y^n_0=R_ny,\quad t \geq 0
\end{cases}
\end{align*}
where $R_n=n(n-A)^{-1}$ for $n\in\mathbb{N}$ with $n > \alpha + \beta + \sqrt{L_F} =: \lambda$. 
By \eqref{eq: semigroup quasi contractive} we find for each $n \geq 1 + \lambda$ the inequality
\[
 \| R_nz \|_H \leq \frac{n}{n - \lambda}\|z\|_H 
 \leq (1 + \lambda)\|z \|_H.
\]
By classical properties of the resolvent (see \cite[Lemma 3.2]{MR710486}),
one clearly has $R_nz \to z$ as $n \to \infty$ in $H$ . Moreover, by properties  of the Yosida approximation of mild solutions of SPDEs (compare e.g. with Appendix A2 in  \cite{9783319128535} or Section 2 in \cite{Albeverio2017}) we have
\[
 \lim_{n \to \infty} \E\left[ \sup_{t \in [0,T]} \| X_t^n - X_t \|_H^2 + \sup_{t \in [0,T]} \| Y_t^n - Y_t \|_H^2 \right] = 0, \qquad \forall T > 0
\]
and hence there exists a subsequence (which is again denoted by $n$) such that $X_t^n \longrightarrow X_t$ and $Y_t^n \longrightarrow Y_t$ hold a.s. for each $t \geq 0$. Following a  method  proposed in  \cite{Albeverio2017} we verify that sufficient conditions are satisfied to apply the  generalized It\^{o}-formula 
from Theorem \ref{GenIto} to the 
function $F(t,z):= \mathrm{e}^{\e t}\| z \|_H^2$, where $\e = 2 \alpha - L_{\sigma} - L_{\gamma}$ is given by \eqref{definition epsilon}:
\begin{align*}
X_t^n-Y_t^n &= R_n (x-y)+\int_0^{t}\left\{ A(X_s^n-Y_s^n)+  R_n (F(X_s^n)-F(Y_s^n))\right\}ds
\\ &+ \int_0^{t}R_n( \sigma(X_s^n)-\sigma(Y_s^n))dW_s
+ \int_0^t\int_E R_n( \gamma(X_{s}^n,\nu)-\gamma(Y_{s}^n,\nu))\widetilde{N}(ds,d\nu).
\end{align*} 
Observe that, by condition (A1) and \eqref{abstractgrowthsforjumpcoeff}, one has
\begin{align*}
&\ \int_0^t\int_E  \| R_n( \gamma(X_s^n,\nu)-\gamma(Y_s^n,\nu))\|_H^2  \mu(d\nu)ds
\\ &\qquad \qquad +  \int_0^t\int_E \| R_n( \gamma(X_s^n,\nu)-\gamma(Y_s^n,\nu))\|_H^4  \mu(d\nu)ds
\\ &\leq (1+ \lambda)^2 \int_0^t \int_E \| \gamma(X_s^n,\nu) - \gamma(Y_s^n,\nu) \|_H^2 \mu(d\nu)ds 
\\ &\ \ \ + 8 (1 + \lambda)^4 \int_0^t \int_E \left( \| \gamma(X_s^n,\nu)\|_H^4 + \| \gamma(Y_s^n, \nu) \|_H^4 \right) \mu(d\nu) ds
\\ &\leq L_{\gamma} (1 + \lambda)^2 \int_0^t \| X_s^n - Y_s^n \|_H^2 ds 
\\ &\qquad \qquad + 16 (1+\lambda)^4 t \sup_{z \in H} \int_E \| \gamma(z,\nu)\|_H^4 \mu(d\nu) < \infty.
\end{align*}
Thus we can apply the generalized It\^{o}-formula 
from Theorem \ref{GenIto} 
and obtain (similar to   (3.5) in \cite{Albeverio2017})
\begin{align}
& \mathrm{e}^{\e t} \|X_t^n-Y_t^n\|_H^2-\|R_n(x-y)\|_H^2\notag 
\\ &= \int_0^t \langle 2 \mathrm{e}^{\e s}  (X_s^n-Y_t^n), R_n(\sigma(X_s^n)-\sigma(Y_s^n)) dW_s \rangle_H\notag
\\ & +  \int_0^t \mathrm{e}^{\e s} \left[\e  \| X_s^n-Y_s^n\|_H^2 +\mathcal{L}_n(\|\cdot\|_H^2)(X_s^n,Y_s^n) \right]ds \notag
\\ &+ \int_0^{t}\int_E \mathrm{e}^{\e s} \left[ \| X_{s}^n-Y_{s}^n+R_n(\gamma(X_{s}^n,\nu)-\gamma(Y_{s}^n,\nu))\|_H^2- \| X_{s}^n-Y_{s}^n\|_H^2 \right]\widetilde{N}(ds, d\nu),\label{IToYOSIDA}
\end{align}
where we used, for $z,w \in D(A)$, the notation 
\begin{align*}
 \mathcal{L}_n(\|\cdot\|_H^2)(z,w)
 &:= 2 \langle z-w , A (z-w) + R_n( F(z)-F(w))\rangle_H + \|R_n (\sigma(z)-\sigma(w))\|_{L_2^0(H)}^2
 \\ &\ \ \ + \int_E \|R_n(\gamma(z,\nu)-\gamma(w,\nu))\|_H^2 \mu(d\nu).\notag
\end{align*}
Taking expectations in \eqref{IToYOSIDA} yields
\begin{align}
& \mathrm{e}^{\e t} \mathbb{E}\left[\|X_t^n-Y_t^n\|_{H}^2 \right]- \mathbb{E} \left[\|R_n(x-y)\|_{H}^2 \right]\notag 
\\ &= \mathbb{E}\left[  \int_0^t \mathrm{e}^{\e s} \left(\e \| X_s^n-Y_s^n\|_{H}^2 +\mathcal{L}_n(\|\cdot\|_{H}^2)(X_s^n,Y_s^n) \right)ds \right].
\end{align}
Lemma \ref{L: Lyapunovfunction} yields
\begin{align*}
& \mathrm{e}^{\e t} \mathbb{E} \left[\|X_t^n-Y_t^n\|_{H}^2 \right]-\mathbb{E} \left[\|R_n(x-y)\|_{H}^2\right]- 2(\alpha+\beta) \int_0^t \mathrm{e}^{\e s} \mathbb{E}\left[\|P_1X_s^n-P_1Y_s^n\|_H^2\right]ds\notag 
\\ &\leq \mathbb{E}\left[  \int_0^t \mathrm{e}^{\e s} (-\mathcal{L}(\|\cdot\|_{H}^2)(X_s^n,Y_s^n)+\mathcal{L}_n(\|\cdot\|_{H}^2)(X_s^n,Y_s^n))ds \right]. 
\end{align*}
Below we prove that the right-hand-side tends to zero as $n \to \infty$, which would imply the assertion of this theorem. To prove the desired convergence to zero we apply 
the generalized Lebesgue Theorem (see \cite[Theorem 7.1.8]{9783319128535}).
For this reason we have to prove that  
\begin{equation}
\label{L: LConvergestoLn2}
 \mathcal{L}(\|\cdot\|_{H}^2)(X_s^n,Y_s^n)- \mathcal{L}_n(\|\cdot\|_{H}^2)(X_s^n,Y_s^n)\to 0
\end{equation}
holds a.s. for each $s>0$ as $n\to \infty$ 
and, moreover, there exists a constant $C > 0$ such that
\begin{align}\label{eq: generalized lebesgue upper bound}
 |\mathcal{L}(\|\cdot\|_{H}^2)(X_s^n,Y_s^n)- \mathcal{L}_n(\|\cdot\|_{H}^2)(X_s^n,Y_s^n)|
 \leq C \| X_s^n - Y_s^n \|_H^2.
\end{align}
We start with the proof of \eqref{L: LConvergestoLn2}.
Denote $F_s^n:=F(X_s^n)-F(Y_s^n)$, $\sigma_s^n:= \sigma(X_s^n)- \sigma(Y_s^n)$ and  $\gamma_s^n(\nu):=\gamma(X_s^n,\nu)-\gamma(Y_s^n,\nu)$ and analogously $F_s:=F(X_s)-F(Y_s)$, $\sigma_s:= \sigma(X_s)- \sigma(Y_s)$ and $\gamma_s(\nu):=\gamma(X_s,\nu)-\gamma(Y_s,\nu)$ for each $n\in\mathbb{N}$, $s\geq 0$ and $\nu\in E$. Then
\begin{align*}
& \vert (\mathcal{L}(\|\cdot\|_H^2)(X_s^n,Y_s^n)- \mathcal{L}_n(\|\cdot\|_H^2)(X_s^n,Y_s^n))\vert
\\ &\leq 2 \vert \langle X_s^n-Y_s^n,  F_s^n-R_nF_s^n\rangle_H \vert
+  \vert\| \sigma_s^n\|_{L_2^0}^2-\| R_n\sigma_s^n\|_{L_2^0}^2\vert
\\ &+ \left| \int_E \| \gamma_s^n(\nu)\|_H^2-\| R_n\gamma_s^n(\nu)\|_H^2 \mu(d\nu) \right| 
\\ &=: I_1 + I_2 + I_3.
\end{align*}
For the first term $I_1$ we estimate
\begin{align*}
 I_1 &\leq 2\|X_s^n-Y_s^n\|_H\|F_s^n-R_n F_s^n\|_H
\\ &\leq 2 \|X_s^n-Y_s^n\|_H \left(\|F_s^n-F_s\|_H + \|F_s-R_nF_s\|_H
+\|R_n F_s-R_n F_s^n\|_H \right)
\\ &\leq 2\|X_s^n-Y_s^n\|_{H} \left(\|F_s^n-F_s\|_H + \|F_s-R_nF_s\|_H + (1+\lambda)\| F_s- F_s^n\|_H \right).
\end{align*}
Using that $X_s^n \to X_s$ and $Y_s^n \to Y_s$ as a.s. for some subsequence (also denoted by $n$),
we easily find that the right-hand side tends to zero.
The convergence of the second term follows from
\begin{align*}
 I_2 &= \left\vert\| \sigma_s^n\|_{L_2^0}-\| R_n\sigma_s^n\|_{L_2^0}\right \vert \left(\| \sigma_s^n\|_{L_2^0}+\| R_n\sigma_s^n\|_{L_2^0}\right)
\\ &\leq (2 + \lambda)\sqrt{L_{\sigma}} \| \sigma_s^n- R_n\sigma_s^n\|_{L_2^0} \| X_s^n-Y_s^n\|_H 
\\ &\leq (2 + \lambda)^2\sqrt{L_{\sigma}}\| X_s^n-Y_s^n\|_H  
\left(\| \sigma_s^n-\sigma_s\|_{L_2^0}+\|\sigma_s-R_n\sigma_s\|_{L_2^0} + \|\sigma_s- \sigma_s^n\|_{L_2^0} \right).
\end{align*}
It remains to show the convergence of the third term. First, observe 
\begin{align*}
 I_3 &\leq (2 + \lambda)\int_E \| \gamma_s^n(\nu)-  R_n\gamma_s^n(\nu)\|_H \|  \gamma_s^n(\nu)\|_H \mu(d\nu)
\\ &\leq (2 + \lambda)\int_E \bigg(\| \gamma_s^n(\nu)-\gamma_s(\nu)\|_H  +\|\gamma_s(\nu)-R_n\gamma_s(\nu)\|_H
   \\ &\qquad \qquad \qquad +\|R_n\gamma_s(\nu)-  R_n\gamma_s^n(\nu)\|_H \bigg)  \|  \gamma_s^n(\nu)\|_{H} \mu(d\nu)
 \\ &\leq (2 + \lambda) \left( \int_E \| \gamma_s^n(\nu)\|_H^2 \mu(d\nu)\right)^{\frac{1}{2}} \bigg[ \left(\int_E \| \gamma_s^n(\nu)-\gamma_s(\nu)\|_H^2 \mu(d\nu) \right)^{\frac{1}{2}} 
 \\ &\qquad + \left(\int_E \|\gamma_s(\nu)-R_n\gamma_s(\nu)\|_H^2  \mu(d\nu) \right)^{\frac{1}{2}}
 + \left(\int_E \|R_n\gamma_s(\nu)-  R_n\gamma_s^n(\nu)\|_H^2  \mu(d\nu)\right)^{\frac{1}{2}} \bigg]
 \\ &\leq \sqrt{2}(2 + \lambda)^2 L_{\gamma} \| X_s^n - Y_s^n\|_H \left(\|X_s^n-X_s\|_{H}+\|Y_s^n-Y_s\|_H \right)
 \\ &\ \ \ + (2 + \lambda) \sqrt{L_{\gamma}} \| X_s^n - Y_s^n \|_H 
 \left(\int_E \|\gamma_s(\nu)-R_n\gamma_s(\nu)\|_{H}^2 \mu(d\nu)  \right)^{\frac{1}{2}}
 \\ &= I^1_3 + I_3^2
\end{align*}
where the last inequality follows from condition (A1) combined with the inequality \begin{align*}
 &\ \| R_n \gamma_s(\nu) - R_n\gamma^n_s(\nu)\|_H^2 
 \\ &\leq (1 + \lambda)^2 \| \gamma_s(\nu) - \gamma_s^n(\nu) \|_H^2
 \\ &\leq 2(1 + \lambda)^2 \left( \| \gamma(X_s,\nu) - \gamma(Y_s,\nu)\|_H^2 + \| \gamma(X_s^n,\nu) - \gamma(Y_s^n, \nu)\|_H^2 \right).
\end{align*}
The first expression $I_1^1$ clearly tends to zero as $n \to \infty$. For the second expression $I_3^2$ we use the inequality
$\|\gamma_s(\nu)-R_n\gamma_s(\nu)\|_{H}^2
\leq 2(2 + \lambda)^2 \| \gamma_s(\nu)\|_H^2$
so that dominated convergence theorem is applicable, which shows that $I_3^2 \to 0$ as $n \to \infty$ a.s.. This proves \eqref{L: LConvergestoLn2}. Concerning \eqref{eq: generalized lebesgue upper bound}, we find that
\begin{align*}
&\ \vert (\mathcal{L}(\|\cdot\|_H^2)(X_s^n,Y_s^n)- \mathcal{L}_n(\|\cdot\|_H^2)(X_s^n,Y_s^n))\vert
\\ &\leq  2 \vert \langle X_s^n-Y_s^n,  F_s^n-R_nF_s^n\rangle_H \vert
+ \vert\| \sigma_s^n\|_{L_2^0(H)}^2-\| R_n\sigma_s^n\|_{L_2^0(H)}^2\vert
\\ &\ \ \ + \left\vert \int_E \| \gamma_s^n(\nu)\|_H^2-\| R_n\gamma_s^n(\nu)\|_H^2 \mu(d\nu) \right\vert 
\\ &\leq 2 ( 2 + \lambda) \| X_s^n-Y_s^n\|_H \|F_s^n\|_H + \left(1 + (1 + \lambda)^2\right) \left[ \| \sigma_s^n\|_{L_2^0(H)}^2
+ \int_E \| \gamma_s^n(\nu)\|_H^2 \mu(d\nu) \right]
\\ &\leq 2 ( 2 + \lambda) L_F \| X_s^n-Y_s^n\|_H^2 + \left( 1 + (1 + \lambda)^2 \right)(L_{\sigma} + L_{\gamma}) \|X_s^n-Y_s^n\|_H^2.
\end{align*}
Hence the generalized Lebesgue Theorem 
is applicable, and thus the assertion of this theorem is proved.
\end{proof}
Note that condition \eqref{abstractgrowthsforjumpcoeff} is used to guarantee that the It\^{o}-formula \ref{GenIto} for Hilbert space valued jump diffusions can be applied for $(x,t)\to \mathrm{e}^{t\varepsilon}\| x \|_H^2$. The assertion of Proposition \ref{T: Lipschitz continuity} is also true when $\varepsilon \leq 0$, but will be only applied for the case when $\varepsilon > 0$.

\section{Convergence to limiting distribution}

\subsection{The Case of Vanishing Coefficients}

It follows from Proposition \ref{T: Lipschitz continuity}
that, for $\e > 0$, one has an estimate on the $L^2$-norm of the difference $X_t^x - X_t^y$. Such an estimate alone does neither imply the existence nor uniqueness of an invariant distribution. 
However, if the coefficients $F, \sigma, \gamma$ 
vanish at $H_1$, then we may characterize the limiting distributions in $L^2$.
\begin{Theorem}\label{thm: HJMcase abstract}
Suppose that (GDC) holds with a projection operator $P_1$, (A1), \eqref{abstractgrowthsforjumpcoeff}, \eqref{definition epsilon} are satisfied, that $(S(t))_{t \geq 0}$ leaves $H_0 := \mathrm{ran}(I - P_1)$ invariant, and that $\mathrm{ran}(P_1) \subset \mathrm{ker}(A)$.
Moreover, assume that 
\begin{align}\label{A: Vanishing Coefficients}
    P_1F\equiv 0,\ P_1\sigma \equiv 0,\ P_1\gamma \equiv 0.
\end{align}
Fix $x \in L^2(\Omega, \mathcal{F}_0, \P;H)$ and suppose that
\begin{equation}\label{Coefficients vanish in H1}
 F(P_1x) = 0, \ \sigma(P_1x) = 0, \ \gamma(P_1x,\cdot) = 0
\end{equation}
holds for this fixed choice of $x$. Then 
\[
 \E\left[ \| X_t^x - P_1x \|_H^2 \right] \leq \mathrm{e}^{-\varepsilon t}\mathbb{E}\left[\|(1-P_1) x\|^2_H \right].
\] 
In particular, let $\rho$ be the law of $x \in L^2(\Omega, \mathcal{F}_0, \P; H)$ and $\rho_1$ be the law of $P_1x$, respectively. Then $\rho_1$ is an invariant measure.
\end{Theorem}
\begin{proof}
Fix $x \in L^2(\Omega, \mathcal{F}_0, \P;H)$ with property \eqref{Coefficients vanish in H1}
and set $P_0 = I - P_1$. Since $(S(t))_{t \geq 0}$ leaves $H_0$ invariant
we find that $S(t)P_0 = P_0S(t)$ and hence we obtain
$P_1 S(t)P_0 = P_1 P_0S(t) = 0$. Moreover, using \eqref{A: Vanishing Coefficients} we find that
\begin{align*}
 P_1X_t^x = P_1 S(t)x = P_1S(t)P_0x + P_1S(t)P_1x = P_1x
\end{align*}
where we have that $S(t)P_1 = P_1$ due to $\mathrm{ran}(P_1) \subset \mathrm{ker}(A)$.
From this we conclude that $(P_0 X_t^{x})_{t\geq 0}$ satisfies 
\begin{align*}
    P_0X_t^x &= P_0S(t)x + \int_0^t P_0S(t-s)F(X_s^x)ds + \int_0^t P_0S(t-s) \sigma(X_s^x)dW_s
    \\ &\qquad \qquad +\int_0^t \int_E P_0S(t-s)\gamma(X_s^x) \widetilde{N}(ds,d\nu)
    \\ &= S(t)P_0x + \int_0^t S(t-s) P_0F(P_1x + P_0X_t^x)ds + \int_0^t S(t-s) P_0\sigma(P_1x + P_0X_s^x)dW_x
    \\ &\qquad \qquad +\int_0^t \int_E S(t-s)P_0\gamma(P_1x + P_0X_s^x)\widetilde{N}(ds,d\nu)
    \\ &= S(t)P_0x + \int_0^t S(t-s) \widetilde{F}(P_0X_t^x)ds + \int_0^t S(t-s)\widetilde{\sigma}(P_0X_s^x) dW_s
    \\ &\qquad \qquad \int_0^t \int_E S(t-s)\widetilde{\gamma}(P_0X_s^x)\widetilde{N}(ds,d\nu),
\end{align*}
where we have set $\widetilde{F}(y):= P_0 F(P_1x + y),\ \widetilde{\sigma}(y):= P_0 \sigma(P_1x + y)$ and $\widetilde{\gamma}(y,\nu):=P_0\gamma(P_1x + y,\nu)$ for all $y \in H_0$ and $\nu \in E$. Since these coefficients share the same Lipschitz estimates as $F,\sigma$ and $\gamma$, we can apply Proposition \ref{T: Lipschitz continuity} to the process $(P_0X_t^x)_{t \geq 0}$ obtained from the above auxiliary SPDE, we obtain
\begin{align*}
    \mathbb{E}[ \| X_t^x - P_1x \|_H^2 ]
    = \mathbb{E}[ \| P_0 X_t^x \|_H^2 ] 
    = \mathbb{E}[ \| P_0 X_t^x - P_0X_t^0\|_H^2 ]
    \leq \mathrm{e}^{-\varepsilon t}\mathbb{E}[\|P_0 x \|^2_H],
\end{align*}
where we have used that $P_0X_t^{0} = 0$ due to \eqref{Coefficients vanish in H1}.
\end{proof}
This theorem can be applied, for instance, to the Heath-Jarrow-Morton-Musiela equation, see Section 6.
Below we discuss two simple examples showing that, in general, conditions $\mathrm{ran}(P_1) \subset \mathrm{ker}(A)$ and \eqref{A: Vanishing Coefficients} cannot be omitted.
\begin{Example}
 Let $H = \R^2$ and $(W_t)_{t \geq 0}$ be a 2-dimensional standard Brownian motion.
\begin{enumerate}
    \item[(a)] Let $X_t = (X_t^1, X_t^2) \in H = \R^2$ be given by
    \begin{align*}
        dX_t = \begin{pmatrix} -1 & 0 \\ 0 & 0 \end{pmatrix}X_tdt + \begin{pmatrix}1 & 0 \\ 0 & 1 \end{pmatrix}dW_t.
    \end{align*}
    Then (GDC) holds for $P_1 = \begin{pmatrix} 0 & 0 \\ 0 & 1 \end{pmatrix}$,
    and $H_1 = \mathrm{ran}(P_1) = \{0\} \times \R = \mathrm{ker}(A)$.
    The semigroup is given by $S(t) = \begin{pmatrix} \mathrm{e}^{-t} & 0 \\ 0 & 1 \end{pmatrix}$ and leaves $H_0 = \mathrm{ran}(I - P_1) = \R \times \{0\}$ invariant.
    Finally we have $\e = 2 > 0$ and \eqref{A: Vanishing Coefficients}, while
    while condition \eqref{A: Vanishing Coefficients} is not satisfied.
    By $X_t^2 = W_t^2$,
    it is clear that $X_t^2$ does not have a limiting distribution.
    Hence also $X_t$ cannot have a limiting distribution.
    \item[(b)] Let $Y_t = (Y_t^1, Y_t^2) \in H = \R^2$ be the solution of
    \begin{align*}
        dY_t = \begin{pmatrix} -1 & 1 \\ 0 & 1 \end{pmatrix}Y_tdt + \begin{pmatrix}1 & 0 \\ 0 & 0 \end{pmatrix}dW_t.
    \end{align*}
    Then (GDC) holds with $P_1$ being the projection onto the second coordinate.
    Hence \eqref{A: Vanishing Coefficients} holds, while $\mathrm{ran}(P_1) \subset \mathrm{ker}(A)$ is not satisfied. Since $Y_t^2 = \mathrm{e}^t Y_0^2 + \int_0^t \mathrm{e}^{t-s}dW_s^2$ it is clear that $Y_t^2$ does not have a limiting distribution. Hence also $Y_t$ cannot have a limiting distribution.
\end{enumerate}
\end{Example}

\subsection{The General Case}

In Theorem \ref{thm: HJMcase abstract} we have assumed \eqref{A: Vanishing Coefficients}, \eqref{Coefficients vanish in H1}, and that $(S(t))_{t \geq 0}$ leaves $H_0$ invariant. 
Below we continue with the more general case.
Namely, for the projection operator $P_1$ given by condition (GDC) we decompose the process $X_t^x$ according to $X_t^x = P_0X_t^x + P_1 X_t^x$, with $P_0 = I - P_1$ and suppose that:
\begin{enumerate}\label{Assumption 4}
    \item[(A2)] The process $P_1X_t^x$ is deterministic of the form
    \begin{align}\label{eq: H1 component deterministic}
     P_1X_t^{x} = P_1S(t)P_1x + \int_0^t P_1S(t-s)P_1F(P_1X_t^{x})ds.
    \end{align}
\end{enumerate}
Our next condition imposes a control on this solution: 
\begin{enumerate}\label{Assumption 5}
 \item[(A3)] For each $x \in H_1 = \mathrm{ran}(P_1)$ there exists $\widetilde{X}_{\infty}^x \in H_1$ and constants $C(x) > 0$, $\delta(x) \in (0,|\e|)$ such that
 \[
 \| P_1X_t^{x} - \widetilde{X}_{\infty}^{x}\|_H^2
 \leq C(x)\mathrm{e}^{-\delta(x) t}, \qquad t \geq 0.
 \]
\end{enumerate} 
Note that, if $P_1 F(P_1 \cdot) = 0$ then condition (A3) reduces to a condition 
on the limiting behavior of the semigroup $(S(t))_{t \geq 0}$ when restricted to $H_1 = \mathrm{ran}(P_1)$. In such a case condition (A3) is, for instance, satisfied if $H_1 \subset \mathrm{ker}(A)$. The following is our main result for this case.
\begin{Theorem}\label{theorem: existence limits}
 Suppose that condition (GDC) holds for some projection operator $P_1$, that conditions (A1) -- (A3), \eqref{abstractgrowthsforjumpcoeff} and \eqref{definition epsilon} are satisfied. Then the following assertions hold:
 \begin{enumerate}
  \item[(a)] For each $x \in H$ there exists an invariant measure $\pi_{\delta_x} \in \mathcal{P}_2(H)$ for the Markov semigroup $(p_t)_{t \geq 0}$ and a constant $K(\alpha,\beta,\e,h) > 0$ such that
 \[
  \mathrm{W}_2(p_t(x,\cdot), \pi_{\delta_x}) \leq K(\alpha, \beta, \e, x)\mathrm{e}^{- \frac{\delta(x)}{2} t}, \qquad t \geq 0.
 \]
  \item[(b)] Suppose, in addition to the conditions of (A3), that there are constants $\delta$ and $C$, such that
  \begin{align}\label{eq:04}
   \delta(x) \geq \delta > 0 \ \ \text{ and } \ \ C(x) \leq C(1 + \| x \|_H)^4, \qquad x \in H.
  \end{align}
  Then, for each $\rho \in \mathcal{P}_2(H)$, there exists an invariant measure $\pi_{\rho} \in \mathcal{P}_2(H)$ for the Markov semigroup $(p_t)_{t \geq 0}$ and a constant $K(\alpha,\beta,\e) > 0$ such that
 \[
  \mathrm{W}_2(p_t^*\rho, \pi_{\rho}) \leq K(\alpha, \beta, \e)\int_{H}(1 + \|x\|_H)^2 \rho(dx)\mathrm{e}^{- \frac{\delta}{2} t}, \qquad t \geq 0.
 \]
 \end{enumerate}
\end{Theorem} 
The proof of this theorem relies on the key stability estimate formulated in Proposition \ref{T: Lipschitz continuity}
and is given at the end of this section.
So far we have only shown the existence of invariant measures parametrized by the initial state of the process.
However, under the given conditions it can also be shown that $\pi_{\delta_x}$ as well as $\pi_{\rho}$ depend only on the $H_1$ part of $x$ or $\rho$, respectively.
\begin{Corollary}\label{T: Uniqueness on affine Subspaces}
 Suppose that condition (GDC) holds for some projection operator $P_1$, that conditions (A1) -- (A3), \eqref{abstractgrowthsforjumpcoeff} and \eqref{definition epsilon} are satisfied.
Then the following assertions hold:
\begin{enumerate}
    \item[(a)] Let $x,y \in H$ be such that $P_1x = P_1y$. Then $\pi_{\delta_x} = \pi_{\delta_y}$.
    \item[(b)] Suppose, in addition, that \eqref{eq:04} holds. Let $\rho, \widetilde{\rho} \in \mathcal{P}_2(H)$ be such that $\rho \circ P_1^{-1} = \widetilde{\rho} \circ P_1^{-1}$. 
    Then $\pi_{\rho} = \pi_{\widetilde{\rho}}$.
\end{enumerate}
\end{Corollary} 
Next we turn to a proof of Theorem \ref{theorem: existence limits} and Corollary \ref{T: Uniqueness on affine Subspaces}.

\subsection{Construction of a coupling}
Let $x \in H$ and let $(X_t^x)_{t \geq 0}$ be the unique mild solution to \eqref{SPDE mild formulation}. Below we construct for given $\tau \geq 0$ a coupling for the law of $(X_t^x, X_{t+\tau}^x)$. Let $(Y_{t}^{x,\tau})_{t \geq 0}$ be the unique mild solution to the SPDE
\begin{align}\label{SPDE mild formulation coupling}
    Y_t^{x,\tau} &= S(t)x + \int_0^t S(t-s)F(Y_s^{x,\tau})ds + \int_0^t S(t-s)\sigma(Y_s^{x,\tau})dW^{\tau}_s
    \\ \notag &\ \ \ + \int_0^t \int_E S(t-s) \gamma(Y_{s}^{x,\tau}, \nu) \widetilde{N}^{\tau}(ds,d\nu), \qquad t \geq 0,
\end{align}
where $W^{\tau}_s = W_{\tau + s} - W_{\tau}$, $\widetilde{N}^{\tau}(ds,d\nu) = \widetilde{N}(ds + \tau, d\nu) - \widetilde{N}(ds,d\nu)$ is a $Q$-Wiener process
and a Poisson random measure with respect to the filtration 
$(\mathcal{F}_s^{\tau})_{s \geq 0}$ defined by $\mathcal{F}_s^{\tau} = \mathcal{F}_{s + \tau}$.
\begin{Lemma}\label{lem: coupling}
 Suppose that (GDC), (A1), \eqref{abstractgrowthsforjumpcoeff} 
and \eqref{definition epsilon} are satisfied. 
Then for each $x \in H$ and $t, \tau \geq 0$ the following assertions hold:
 \begin{enumerate}
     \item[(a)] $Y_{t}^{x,\tau}$ has the same law as $X_t^{x}$.
     \item[(b)] It holds that
     \begin{align*}
         \E\left[ \| Y_t^{x,\tau} - X_{t+\tau}^x \|_H^2 \right]
         &\leq \mathrm{e}^{-\e t} \E\left[ \| x - X_{\tau}^x \|_H^2 \right]
         \\ &\ \ \ + 2 ( \alpha + \beta) \int_0^t \mathrm{e}^{- \e(t-s)} \E\left[ \| P_1Y_s^{x,\tau} - P_1X_{s + \tau}^x \|_H^2 \right] ds.
     \end{align*}
 \end{enumerate}
\end{Lemma}
\begin{proof}
 \textit{(a)} Since \eqref{SPDE mild formulation} has a unique solution it follows from the 
 Yamada-Watanabe Theorem (see \cite{MR2336594}) that also uniqueness in law holds for this equation.
 Since the driving noises $N^{\tau}$ and $W^{\tau}$ in \eqref{SPDE mild formulation coupling} 
 have the same law as $N$ and $W$ from \eqref{SPDE mild formulation}, it follows that 
 the unique solution to \eqref{SPDE mild formulation coupling} has the same law as the solution to \eqref{SPDE mild formulation}.
 This proves the assertion.
 
 \textit{(b)} Set $X_{t}^{x,\tau} := X_{t+ \tau}^x$, then by direct computation we find that
 \begin{align*}
     X_{t}^{x,\tau} &= S(t)S(\tau)x + \int_0^{t + \tau} S(t+\tau - s)F(X_s^x) ds 
     + \int_0^{t + \tau} S(t + \tau-s)\sigma(X_s^{x})dW_s
     \\ &\ \ \ + \int_0^{t + \tau}\int_E S(t+\tau - s)\gamma(X_s^{x},\nu)\widetilde{N}(ds,d\nu)
     \\ &= S(t)S(\tau)x + S(t)\int_{0}^{\tau}S(\tau - s)F(X_s^x)ds
     + S(t)\int_0^{\tau}S(\tau-s)\sigma(X_s^x)dW_s
     \\ &\ \ \ + S(t)\int_0^{\tau}\int_E S(\tau-s)\gamma(X_s^x,\nu)\widetilde{N}(ds,d\nu)
     \\ &\ \ \ + \int_{\tau}^{t+\tau}S(t + \tau - s)F(X_s^x)ds
     + \int_{\tau}^{t+\tau}S(t+\tau-s)\sigma(X_s^x)dW_s
     \\ &\ \ \ + \int_{\tau}^{t+\tau}\int_E S(t+\tau-s)\gamma(X_s^x,\nu)\widetilde{N}(ds,d\nu)
     \\ &= S(t)X_{0}^{x,\tau} + \int_0^{t}S(t-s)F(X_{s}^{x,\tau})ds
     + \int_0^t S(t-s)\sigma(X_s^{x,\tau})dW^{\tau}_s 
     \\ &\ \ \ + \int_0^{t}\int_E S(t-s)\gamma(X_s^{x,\tau},\nu)\widetilde{N}^{\tau}(ds,d\nu),
 \end{align*}
 where in the last equality we have used, for appropriate integrands $\Phi(s,\nu)$ and $\Psi(s)$, that
 \begin{align*}
     \int_{\tau}^{\tau + t}\Psi(s)dW_s &= \int_0^t \Psi(s+\tau)dW_s^{\tau},
     \\ \int_{\tau}^{\tau + t}\int_E \Phi(s,\nu)\widetilde{N}(ds,d\nu) &= \int_0^t \int_E \Phi(s+\tau,\nu)\widetilde{N}^{\tau}(ds,d\nu).
 \end{align*}
 Hence $(X_t^{x,\tau})_{t \geq 0}$ also solves \eqref{SPDE mild formulation coupling} with $\mathcal{F}_0^{\tau} = \mathcal{F}_{\tau}$ and initial condition $X_0^{x,\tau} = X_{\tau}^x$. Consequently, the assertion follows from Proposition \ref{T: Lipschitz continuity} applied to $X_t^{x,\tau}$ and $Y_t^{x,\tau}$.
\end{proof}

\subsection{Proof of Theorem \ref{theorem: existence limits}}

\begin{proof}[Proof of Theorem \ref{theorem: existence limits}]
Fix $x \in H$ and recall that $p_t(x,\cdot)$ denotes the transition probabilities of the Markov process obtained from \eqref{SPDE mild formulation}. Below we prove that $(p_t(x,\cdot))_{t \geq 0} \subset \mathcal{P}_2(H)$ is a Cauchy sequence with respect to the Wasserstein distance $\mathrm{W}_2$. Fix $t, \tau \geq 0$. We treat the cases $\tau \in (0,1]$ and $\tau > 1$ separately.

Case $0 < \tau \leq 1$: Then using the coupling lemma \ref{lem: coupling}.(b) yields
\begin{align*}
    \mathrm{W}_2(p_{t+ \tau}(x,\cdot), p_t(x,\cdot)) &\leq \left( \E\left[ \| Y_{t}^{x,\tau} - X_{t+\tau}^x \|_H^2 \right] \right)^{1/2}
    \\ &\leq \mathrm{e}^{- \frac{\e}{2}t} \left( \E\left[ \| X_{\tau}^x - x \|_H^2 \right] \right)^{1/2}
    \\ &\ \ \ + \sqrt{2(\alpha + \beta)} \left( \int_0^t \mathrm{e}^{- \e(t - s)} \E\left[ \| P_1Y_s^{x,\tau} - P_1X_{s + \tau}^x \|_H^2 \right] ds \right)^{1/2}
    \\ &=: I_1 + I_2.
\end{align*}
The first term $I_1$ can be estimated by
\begin{align*}
    I_1 &\leq \mathrm{e}^{- \frac{\e}{2}t} \sup_{s \in [0,1]} \left( \E\left[ \| X_s^x - x \|_H^2 \right] \right)^{1/2}.
\end{align*}
To estimate the second term $I_2$ we first observe that by  condition (A2) we have $P_1Y_s^{x,\tau} = P_1X_s^x$ and hence by condition (A3) one has for each $s \geq 0$ that
\begin{align*}
 \E\left[ \| P_1Y_s^{x,\tau} - P_1X_{s + \tau}^{x} \|_H^2 \right]
 &\leq 2 \| P_1 Y_{s}^{x,\tau} - \widetilde{X}_{\infty}^{x}\|_H^2 + 2 \| P_1 X_{s+\tau}^x - \widetilde{X}_{\infty}^{x} \|_H^2
 \\ &\leq 4C(x) \mathrm{e}^{-\delta(x) s}.
\end{align*}
This readily yields
\begin{align*}
   &\ \int_0^{t} \mathrm{e}^{- \e( t - s)} \E\left[ \| P_1Y_s^{x,\tau} - P_1X_{s + \tau}^{x} \|_H^2 \right] ds
\\    &\leq 4C(x)\int_0^{t}\mathrm{e}^{- \e (t - s)} \mathrm{e}^{- \delta(x)s} ds
    \\ &=  4C(x)\mathrm{e}^{- \e t} \frac{ \mathrm{e}^{(\e - \delta(x))t} - 1}{\e - \delta(x)}
    \\ &\leq 4C(x)\frac{\mathrm{e}^{- \delta(x)t}}{\e - \delta(x)}.
\end{align*}
Inserting this into the definition of $I_2$ gives
\begin{align*}
    I_2 &\leq 2 \sqrt{ \frac{(\alpha + \beta) C(x)}{ \e - \delta(x)}}\mathrm{e}^{- \frac{\delta(x)}{2}t}.
\end{align*}

Case $\tau > 1$: Fix some $N \in \N$ with $\tau < N < 2\tau$ and define a sequence of numbers $(a_n)_{n = 0,\dots, N}$ by
\[
 a_n := \frac{\tau}{N} n, \qquad n = 0, \dots, N.
\]
Then $a_0 = 0$, $a_N = \tau$ and $a_n - a_{n-1} = \frac{\tau}{N} =: \varkappa \in (\frac{1}{2},1)$ for $n = 1, \dots, N$. Hence we obtain from the coupling Lemma \ref{lem: coupling}.(b)
\begin{align*}
 &\mathrm{W}_2(p_{t+\tau}(x,\cdot), p_t(x,\cdot)) 
 \\ &\leq \sum_{n=1}^{N}\mathrm{W}_2( p_{t + a_{n}}(x,\cdot), p_{t+a_{n-1}}(x,\cdot))
 \\ &\leq \sum_{n=1}^{N}\left( \E\left[ \| Y_{t+a_{n-1}}^{x,\varkappa} - X_{t + a_{n-1} + \varkappa}^{x} \|_H^2 \right] \right)^{1/2}
 \\ &\leq \sum_{n=1}^{N} \mathrm{e}^{-\frac{\e}{2}(t+a_{n-1})}\left(\mathbb{E}\left[\| X_{\varkappa}^x - x\|_H^2 \right] \right)^{1/2} 
 \\ &\ \ \ + \sqrt{2(\alpha + \beta)}\sum_{n=1}^{N} \left( \int_0^{t+a_{n-1}} \mathrm{e}^{- \e(t+a_{n-1}-s)}\E\left[ \| P_1Y_s^{x,\varkappa} - P_1X_{s + \varkappa}^{x} \|_H^2 \right] ds \right)^{1/2}
 \\ &=: I_1 + I_2.
\end{align*}
For the first term $I_1$ we use $\varkappa > \frac{1}{2}$
so that 
\[
 \sum_{n=1}^{N}\mathrm{e}^{- \frac{\e}{2}\varkappa (n-1)} \leq \sum_{n=0}^{\infty} \mathrm{e}^{- \frac{\e}{4}n} = \left(1 - \mathrm{e}^{- \frac{\e}{4}}\right)^{-1},
\] 
from which we obtain
\begin{align*}
 I_1 &= \mathrm{e}^{- \frac{\e}{2} t} \sup_{s \in [0,1]}\left( \E[ \| X_{s}^x - x\|_H^2] \right)^{\frac{1}{2}} \sum_{n=1}^{N}\mathrm{e}^{- \frac{\e}{2}\varkappa (n-1)}
 \\ &\leq  \sup_{s \in [0,1]}\left( \E[ \| X_{s}^x - x\|_H^2] \right)^{\frac{1}{2}} \left( 1 - \mathrm{e}^{- \frac{\e}{4}}\right)^{-1} \mathrm{e}^{- \frac{\e}{2}t}.
\end{align*}
To estimate the second term $I_2$ we first observe that by condition (A2) we have $P_1Y_s^{x,\tau} = P_1 X_s^x$ and hence by condition (A3), one has for $s \geq 0$
\begin{align*}
 \E\left[ \| P_1Y_s^{x,\varkappa} - P_1X_{s + \varkappa}^{x} \|_H^2 \right]
 &\leq 2  \| P_1 Y_{s}^{x,\varkappa} - \widetilde{X}_{\infty}^{x}\|_H^2 + 2 \| P_1 X_{s+\varkappa}^x - \widetilde{X}_{\infty}^{x} \|_H^2
 \\ &\leq 4C(x) \mathrm{e}^{-\delta(x) s}.
\end{align*}
Hence we find that
\begin{align*}
   &\ \int_0^{t + a_{n-1}} \mathrm{e}^{- \e( t + a_{n-1} - s)} \E\left[ \| P_1Y_s^{x,\varkappa} - P_1 X_{s + \varkappa}^{x} \|_H^2 \right] ds
\\    &\leq 4C(x)\int_0^{t + a_{n-1}} \mathrm{e}^{- \e (t + a_{n-1} - s)} \mathrm{e}^{- \delta(x)s} ds
    \\ &=  4C(x)\mathrm{e}^{- \e (t + a_{n-1})} \frac{ \mathrm{e}^{(\e - \delta(x))(t + a_{n-1})} - 1}{\e - \delta(x)}
    \\ &\leq 4C(x)\frac{\mathrm{e}^{- \delta(x)t}}{\e - \delta(x)} \mathrm{e}^{- \delta(x) a_{n-1}}
    \\ &\leq 4C(x) \frac{\mathrm{e}^{- \delta(x)t}}{\e - \delta(x)} \mathrm{e}^{- \frac{ \delta(x)}{2} (n-1)}
\end{align*}
where the last inequality follows from $a_{n-1} = \varkappa (n-1) \geq \frac{1}{2}(n-1)$.
From this we readily derive the estimate
\begin{align*}
    I_2 &\leq 2 \sqrt{ \frac{(\alpha + \beta) C(x)}{ \e - \delta(x)}} \left( 1 - \mathrm{e}^{- \frac{\delta(x)}{4}} \right)^{-1}\mathrm{e}^{- \frac{\delta(x)}{2}t}.
\end{align*}
Hence we obtain 
\begin{align}\label{eq: estimate on cauchy property}
 \mathrm{W}_2(p_{t+\tau}(x,\cdot), p_t(x,\cdot)) \leq K(\alpha, \beta, \e, x)\mathrm{e}^{- \frac{\delta(x)}{2}t}, \qquad t,\tau \geq 0,
\end{align}
where the constant $K(\alpha, \beta, \e, x) > 0$ is given by
\[
K(\alpha, \beta, \e, x) = K(\e)(1 + \| x\|_H) + 2 \sqrt{ \frac{(\alpha + \beta) C(x)}{ \e - \delta(x)}} \left( 1 - \mathrm{e}^{- \frac{\delta(x)}{4}} \right)^{-1}
\]
with another constant $K(\e) > 0$.
This implies that, for each $x \in H$, 
$(p_t(x,\cdot))_{t \geq 0}$ has a limit in $\mathcal{P}_2(H)$.
Denote this limit by $\pi_{\delta_x}$. 
Assertion (a) now follows by taking the limit $\tau \to \infty$ in \eqref{eq: estimate on cauchy property} and using the fact that $K(\alpha, \beta, \e, x)$ is independent of $\tau$.

It remains to prove assertion (b). First observe that, using $\delta(x) \geq \delta > 0$ and $C(x) \leq C(1 + \| x \|_H)^4$, we have
\[
 K(\alpha,\beta, \e,x) \leq (1 + \| x \|_H)^2 \widetilde{K}(\alpha,\beta,\e)
\]
for some constant $\widetilde{K}(\alpha,\beta,\e)$. Note that 
\[
 p_t^*\rho(dy) = \int_H p_t(z,dy)\rho(dz)
 \ \ \text{ and } \ \ p_{t+\tau}^*\rho(dy) = \int_H p_{t+\tau}(z,dy)\rho(dz).
\]
Hence using first the convexity of the Wasserstein distance and then \eqref{eq: estimate on cauchy property} we find that
\begin{align*}
    \mathrm{W}_2(p_{t+\tau}^*\rho, p_t^* \rho) 
    &\leq \int_{H} \mathrm{W}_2(p_{t+\tau}(x,\cdot), p_t(x,\cdot)) \rho(dx)
    \\ &\leq \widetilde{K}(\alpha,\beta, \e) \int_H (1 + \|x \|_H)^2 \rho(dx) \cdot \mathrm{e}^{- \frac{\delta}{2}t}.
\end{align*}
Since $\rho \in \mathcal{P}_2(H)$, the assertion is proved.
\end{proof}

\subsection{Proof of Corollary \ref{T: Uniqueness on affine Subspaces}}

\begin{proof}[Proof of Corollary \ref{T: Uniqueness on affine Subspaces}]
 Recall that, by condition (A2) the process $P_1 X_t^x$ is deterministic of the form
 \[
  P_1X_t^x = P_1 S(t)P_1x + \int_0^t P_1S(t-s)F(P_1X_s^x)ds.
 \]
 Since $F$ is globally Lipschitz continuous by condition (A1), it follows that this equation has
 for each $x \in H$ a unique solution. From this we readily conclude that
 $P_1 X_t^{x}=P_1 X_t^{y}$ holds for all $t\geq 0$, provided that $P_1x = P_1y$.
 Hence Proposition \ref{T: Lipschitz continuity} yields for such $x,y$
\begin{equation}\label{Lipschitzcontinuity}
\mathbb{E}\left[\|X_t^{x}-X_t^{y}\|_H^2\right]\leq \mathrm{e}^{-\e t}\| x-y\|_{H}^2, \qquad \forall t\geq 0.
\end{equation} 
Then for each $x,y \in H$ with $P_1x = P_1y$ and each $t \geq 0$ we obtain
\begin{align*}
    \mathrm{W}_2(\pi_{\delta_x}, \pi_{\delta_y}) &\leq \mathrm{W}_2(\pi_{\delta_x}, p_t(x,\cdot))
    + \mathrm{W}_2(p_t(x,\cdot), p_t(y,\cdot)) + \mathrm{W}_2(p_t(y,\cdot), \pi_{\delta_y})
    \\ &\leq \mathrm{W}_2(\pi_{\delta_x}, p_t(x,\cdot))
    + \mathrm{e}^{- \frac{\e}{2}t}\| x - y \|_H + \mathrm{W}_2(p_t(y,\cdot), \pi_{\delta_y}).
\end{align*}
Letting $t \to \infty$ yields $\pi_{\delta_x} = \pi_{\delta_y}$ and hence assertion (a) is proved. 

To prove assertion (b), let $\rho, \widetilde{\rho} \in \mathcal{P}_2(H)$ be such that $\rho \circ P_1^{-1} = \widetilde{\rho} \circ P_1^{-1}$. Then
\begin{align*}
    \mathrm{W}_2(\pi_{\rho}, \pi_{\widetilde{\rho}}) &\leq \mathrm{W}_2( \pi_{\rho}, p_t^* \rho) + \mathrm{W}_2( p_t^*\rho, p_t^* \widetilde{\rho}) + \mathrm{W}_2( p_t^* \widetilde{\rho}, \pi_{\widetilde{\rho}})
\end{align*}
Again, by letting $t \to \infty$, it suffices to prove that
\begin{align}\label{eq:00}
 \limsup_{t \to \infty}\mathrm{W}_2( p_t^*\rho, p_t^* \widetilde{\rho}) = 0.
\end{align}
Let $G$ be a coupling of $(\rho, \widetilde{\rho})$.
Using the convexity of the Wasserstein distance and Proposition \ref{T: Lipschitz continuity} gives
\begin{align*}
    &\ \mathrm{W}_2(\mathcal{P}_t^*\rho, \mathcal{P}_t^* \widetilde{\rho})
    \\ &\leq \int_{H \times H} \mathrm{W}_2(p_t(x,\cdot), p_t(y,\cdot)) G(dx,dy)
    \\ &\leq \int_{H \times H} \left(\E\left[ \| X_t^x - X_t^y \|_H^2 \right]\right)^{1/2} G(dx,dy)
    \\ &\leq \int_{H \times H} \mathrm{e}^{-\frac{\e}{2}t} \| x -y\|_H G(dx,dy)
    \\ &\ \ \ + \sqrt{2(\alpha + \beta)}\int_{H^2} \left( \int_0^{t} \mathrm{e}^{- \e(t-s)}\E\left[ \| P_1X_s^x - P_1X_s^y \|_H^2 \right] ds \right)^{1/2} G(dx,dy)
    \\ &=: I_1 + I_2.
\end{align*}
The first term $I_1$ satisfies 
\[
 I_1 \leq \left( 2 + \int_{H} \| x \|_H^2 \rho(dx) + \int_{H} \| y \|_H^2 \widetilde{\rho}(dy) \right) \mathrm{e}^{- \frac{\e}{2}t}.
\]
For the second term we first use (A2) so that $P_1X_s^x = P_1X_s^{P_1x}$, $P_1X_s^{y} = P_1X_s^{P_1y}$ and hence we find for each $T > 0$ a constant $C(T) > 0$ such that for $t \in [0,T]$
\begin{align*}
 I_2 &= \sqrt{2(\alpha + \beta)}\int_{H_1^2} \left( \int_0^{t} \mathrm{e}^{- \e(t-s)} \| P_1X_s^x - P_1X_s^y \|_H^2  ds \right)^{1/2} G(dx, dy)
 \\ &= C(T) \left(\int_{H^2} \| P_1x - P_1y\|_H^2 G(dx,dy) \right)^{1/2}.
\end{align*}
Let us choose a particular coupling $G$ as follows:
By disintegration we write $\rho(dx) = \rho(x_1,dx_0)(\rho \circ P_1^{-1})(dx_1)$,
$\widetilde{\rho}(dx) = \widetilde{\rho}(x_1,dx_0)(\widetilde{\rho} \circ P_1^{-1})(dx_1) = \widetilde{\rho}(x_1,dx_0)(\rho \circ P_1^{-1})(dx_1)$ where $\rho(x_1,dx_0)$, $\widetilde{\rho}(x_1,dx_0)$ are transition kernels defined on $H_1 \times \mathcal{B}(H_0)$ and we have used that 
$(\rho \circ P_1^{-1})(dx_1) = (\widetilde{\rho} \circ P_1^{-1})(dx_1)$. Then $G$ is, for $A,B \in \mathcal{B}(H)$, given by
\[
 G(A \times B) := \int_{H^2} \1_{A}(x_0, x_1) \1_{B}(y_0, y_1)\rho(x_1,dx_0)\widetilde{\rho}(y_1,dy_0)\widetilde{G}(dx_1,dy_1),
\]
where $\widetilde{G}$ is a probability measure on $H_1^2$ given, for $A_1,B_1 \in \mathcal{B}(H_1)$, by
\[
 \widetilde{G}(A_1 \times B_1) = (\rho \circ P_1^{-1})(A_1 \cap B_1) = \rho\left( \{ x \in H \ | \ P_1x \in A_1 \cap B_1 \} \right).
\]
For this particular choice of $G$ we find that
\begin{align*}
 \int_{H^2} \| P_1 x - P_1y \|_H^2 G(dx,dy)
 &= \int_{H_1^2} \int_{H_0^2}\| x_1 - y_1 \|_H^2 \rho(x_1,dx_0) \widetilde{\rho}(y_1,dy_0)\widetilde{G}(dx_1,dy_1)
 \\ &= \int_{H_1^2} \| x_1 - y_1 \|_H^2 \widetilde{G}(dx_1,dy_1) = 0
\end{align*}
and hence $I_2 = 0$. This proves \eqref{eq:00} and completes the proof.
\end{proof}

\section{Ornstein-Uhlenbeck process on Hilbert space}
Let $H$ be a separable Hilbert space and let
$(Z_t)_{t \geq 0}$ be a $H$-valued L\'evy process on a stochastic basis $(\Omega, \mathcal{F}, (\mathcal{F}_t)_{t \geq 0}, \P)$ with the usual conditions. 
Then, following \cite{A15}, it has characteristic exponent $\Psi$ of L\'evy-Khinchine form, i.e.
\[
 \E\left[ \mathrm{e}^{\mathrm{i} \langle u, Z_t \rangle_H} \right]
 = \mathrm{e}^{t \Psi(u)}, \qquad u \in H, \ \ t > 0,
\] 
with $\Psi$ given by
\begin{align*}
    \Psi(u) &= \mathrm{i} \langle b, u\rangle_H - \frac{1}{2}\langle Qu,u\rangle_H + \int_H \left( \mathrm{e}^{\mathrm{i}\langle u, z\rangle_H} - 1 - \mathrm{i}\langle u,z \rangle_H \1_{ \{ \| z \|_H \leq 1\}} \right) \mu(dz),
\end{align*}
where $b \in H$ denotes the drift, $Q$ denotes the covariance operator being a positive, symmetric, trace-class operator on $H$, and $\mu$ is a L\'evy measure on $H$.
Let $(S(t))_{t \geq 0}$ be a strongly continuous semigroup on $H$. The Ornstein-Uhlenbeck process driven by $(Z_t)_{t \geq 0}$ is the unique mild solution to
\[
 dX^x_t = AX^x_tdt + dZ_t, \qquad X^x_0 = x \in H, \ \ t \geq 0,
\]
where $(A,D(A))$ denotes the generator of $(S(t))_{t \geq 0}$,
i.e. $(X^x_t)_{t \geq 0}$ satisfies
\[
 X^x_t = S(t)x + \int_0^t S(t-s)dZ_s, \qquad t \geq 0.
\]
The characteristic function of $(X^x_t)_{t \geq 0}$ is given by
\[
 \E\left[ \mathrm{e}^{\mathrm{i} \langle u, X^x_t \rangle_H} \right]
 = \exp\left( \mathrm{i}\langle S(t)x, u \rangle_H + \int_0^t \Psi(S(r)^*u)dr \right), \qquad u \in H, \ \ t \geq 0.
\]
For additional properties, references and related results 
we refer to the review article \cite{A15}
where also the existence, uniqueness and properties of invariant measures are discussed. Following these results, the Ornstein-Uhlenbeck process has a unique invariant measure provided that $(S(t))_{t \geq 0}$ is uniformly exponentially stable, that is 
\[
 \exists \alpha > 0, \ M \geq 1: \qquad \| S(t) \|_{L(H)} \leq M\mathrm{e}^{-\alpha t}, \qquad t \geq 0,
\]
and the L\'evy measure $\mu$ satisfies a $\log$-integrability condition for its big jumps
\begin{align}\label{moment condition mu}
 \int_{ \{ \| z \|_H > 1 \} } \log(1 + \|z \|_H) \mu(dz) < \infty.
\end{align}
Below we show that for a uniformly convergent semigroup $(S(t))_{t \geq 0}$ the corresponding Ornstein-Uhlenbeck process may admit multiple invariant measures parameterized by the range of the limiting projection operator of the semigroup. 
\begin{Theorem}\label{Theorem: OU process}
Suppose that $(S(t))_{t \geq 0}$ is uniformly exponentially convergent,
i.e. there exists a projection operator $P$ on $H$ and constants $M \geq 1$, $\alpha > 0$ such that
\begin{align}\label{eq: convergent semigroup}
 \| S(t)x - Px \|_H \leq M \| x \|_H \mathrm{e}^{- \alpha t}, \qquad t \geq 0, x \in H.
\end{align}
Suppose that the L\'evy process satisfies the following conditions:
\begin{enumerate}
    \item[(i)] The drift $b$ satisfies $Pb = 0$.
    \item[(ii)] The covariance operator $Q$ satisfies $PQu = 0$ for all $u \in H$.
    \item[(iii)] The L\'evy measure $\mu$ is supported on $\mathrm{ker}(P)$ and satisfies \eqref{moment condition mu}.
\end{enumerate}
Then 
\[
 X_t^x \longrightarrow Px + X_{\infty}^0, \qquad x \in H,
\]
in law, where $X_{\infty}^0$ is an $H$-valued random variable determined by
\[
 \E\left[ \mathrm{e}^{\mathrm{i} \langle u, X_{\infty}^0 \rangle_H}\right]
 = \exp\left( \int_0^{\infty} \Psi(S(r)^*u)dr \right).
\]
In particular, the set of all limiting distributions for the Ornstein-Uhlenbeck process $(X^x_t)_{t \geq 0}$ is given by $\left\{ \delta_{x} \ast \mu_{\infty} \ | \ x \in \mathrm{ker}(P) \right\}$, where $\mu_{\infty}$ denotes the law of $X_{\infty}^0$.
\end{Theorem}
\begin{proof}
 We first prove the existence of a constant $C > 0$ such that
 \begin{align}\label{eq:01}
  \int_0^{\infty}|\Psi(S(r)^*u)|dr \leq C ( \| u \|_H + \| u \|_H^2), \qquad u \in H.
 \end{align}
 To do so we estimate
 \begin{align*}
  | \Psi(S(r)^*u) | &\leq  | \langle b, S(r)^*u \rangle | + |\langle QS(r)^*u, S(r)^*u \rangle|
 \\ &\ \ \  + \int_{ \{ \| z \|_H \leq 1 \} } \left| \mathrm{e}^{\mathrm{i}\langle S(r)^*u, z\rangle} - 1 - \mathrm{i}\langle S(r)^*u,z \rangle \right| \mu(dz)
  \\ &\ \ \  + \int_{ \{\| z \|_H > 1 \}} \left| \mathrm{e}^{\mathrm{i}\langle S(r)^*u, z\rangle} - 1\right| \mu(dz)
  \\ &= I_1 + I_2 + I_3 + I_4.
 \end{align*}
 We find by \eqref{eq: convergent semigroup} that $\| S(r)x \|_H \leq M\mathrm{e}^{- \alpha r}\| x \|_H$ for all $x \in \mathrm{ker}(P)$ and hence
 \[
  I_1 = | \langle S(r)b, u \rangle | 
  \leq \| u \|_H \| S(r)b \|_H \leq \| u \|_H M \mathrm{e}^{- \alpha r} \| b \|_H.
 \]
 For the second term $I_2$ we use $\mathrm{ran}(Q) \subset \mathrm{ker}(P)$ so that
 \[
  \| S(r)Qu \|_H \leq M\mathrm{e}^{-\alpha r} \|Qu \|_H
  \leq \mathrm{e}^{- \alpha r} \| Q \|_{L(H)} \|u \|_H.
 \]
 This yields $\| QS(r)^* \|_{L(H)} = \| S(r)Q\|_{L(H)} \leq M \mathrm{e}^{- \alpha r} \| Q \|_{L(H)}$ and hence
 \begin{align*}
  I_2 &= |\langle QS(r)^*u, S(r)^*u \rangle|
  \\ &\leq  \| QS(r)^*u \|_H \| S(r)^* u\|_H 
  \\ &\leq M\| u \|_H \| QS(r)^*u \|_{H} 
  \\ &\leq M\| u \|_H^2 \| Q \|_{L(H)}M \mathrm{e}^{-\alpha r}.
 \end{align*}
 For the third term $I_3$ we obtain
 \begin{align*}
  I_3 &\leq C \int_{ \{ \| z \|_H \leq 1\} } | \langle S(r)^*u, z \rangle|^2 \mu(dz)
  \\ &= C \int_{ \{ \| z \|_H \leq 1 \} \cap \mathrm{ker}(P) } | \langle u, S(r)z \rangle |^2 \mu(dz)
  \\ &\leq C\| u\|_H^2 \mathrm{e}^{- \alpha r} \int_{ \{ \| z \|_H \leq 1 \} } \| z \|_H^2 \mu(dz),
 \end{align*}
 where $C > 0$ is a generic constant.
 Proceeding similarly for the last term, we obtain
 \begin{align*}
  I_3 &\leq C \int_{ \{ \| z \|_H > 1 \} }  \min \left\{ 1,  | \langle S(r)^*u, z \rangle| \right\} \mu(dz)
   \\ &\leq C  \int_{ \{ \| z \|_H > 1 \} \cap \mathrm{ker}(P) } \min \left\{ 1, \| u\|_H \mathrm{e}^{- \alpha r} \| z \|_H \right\} \mu(dz)
  \\ &\leq C \| u \|_H \mathrm{e}^{- \alpha r} \left( \mu( \{ \| z \|_H > 1\} ) + \int_{ \{ \|z\|_H > 1\}} \log( 1 + \| z \|_H) \mu(dz) \right),
 \end{align*}
 where we have used, for $a = \| u\|_H \mathrm{e}^{-\alpha r}$, $b = \|z\|_H$, the elementary inequalities
 \begin{align*}
  \min\{1,ab\} &\leq C \log(1 + ab) 
  \\ &\leq C \min\{ \log(1 +a), \log(1+b)\} + C \log(1+a)\log(1+b)
  \\ &\leq C a \left( 1 + \log(1+b) \right),
 \end{align*}
 see \cite[appendix]{FJR19}.
 Combining the estimates for $I_1,I_2,I_3, I_4$ we conclude that \eqref{eq:01} is satisfied. Hence, using 
 \[
  \lim_{t \to \infty} \langle S(t)x, u \rangle = \langle Px,u \rangle
 \]
 we find that
 \begin{align}\label{eq:02}
  \lim \limits_{t \to \infty} \E\left[ \mathrm{e}^{\mathrm{i} \langle u, X_t^x \rangle} \right]
  = \exp\left( \mathrm{i}\langle Px, u \rangle + \int_0^{\infty} \Psi(S(r)^*u)dr \right).
 \end{align}
 Since, in view of \eqref{eq:01}, $u \longmapsto \int_0^{\infty}\Psi(S(r)^*u)dr$ is continuous at $u = 0$, the assertion follows from L\'evy's continuity theorem combined with the particular form of \eqref{eq:02}.
\end{proof}
Below we briefly discuss an application of this result
to a stochastic perturbation of the Kolmogorov equation associated
with a symmetric Markov semigroup.
Let $X$ be a Polish space and $\eta$ a Borel probability measure on $X$. Let $(A,D(A))$ be the generator of a symmetric Markov semigroup $(S(t))_{t \geq 0}$ on $H := L^2(X,\eta)$. Then there exists, for each $f \in D(A)$, a unique solution to the Kolmogorov equation (see, e.g., \cite{MR710486})
\[
 \frac{dv(t)}{dt} = Av(t), \qquad v(0) = f.
\]
Below we consider an additive stochastic perturbation of this equation in the sense of It\^{o}, i.e. the stochastic partial differential equation
\begin{align}\label{Kolmogorov}
 dv(t) = Av(t)dt + dZ_t, \qquad v(0) = f,
\end{align}
where $(Z_t)_{t \geq 0}$ is an $L^2(X,\eta)$-valued L\'evy process with characteristic function $\Psi$. Let $(v(t);f))_{t \geq 0}$ be the unique mild solution to this equation.
\begin{Corollary}
 Suppose that the semigroup generated by $(A,D(A))$ on $L^2(X,\eta)$ satisfies \eqref{eq: convergent semigroup} with the projection operator
 \[
  P_1v = \int_X v(x) \eta(dx),
 \]
 and assume that the L\'evy process $(Z_t)_{t \geq 0}$ satisfies 
 the conditions (i) -- (iii) of Theorem \ref{Theorem: OU process}. Then
 \[
  v(t; f) \longrightarrow \int_X f(x) \eta(dx) + v(\infty), \qquad t \to \infty
 \]
 in law, where $v(\infty)$ is a random variable whose characteristic function is given by 
 \[
 \E\left[ \mathrm{e}^{\mathrm{i} \langle u, v(\infty) \rangle_{L^2}}\right]
 = \exp\left( \int_0^{\infty} \Psi(S(r)^*u)dr \right).
 \]
\end{Corollary}

\section{The Heath-Jarrow-Mortion-Musiela equation}

The Heath-Jarrow-Morton-Musiela equation (HJMM-equation) describes the term structure of interest rates in terms of its forward rate dynamics modelled, for $\beta > 0$ fixed, on the space of forward curves 
\begin{align}\label{FCS}
H_{\beta} &= \left\{ h:\mathbb{R}_+\to\mathbb{R}: h\text{ is absolutely continuous 
and } \| h \|_{\beta} < \infty \right\},
 \\ \notag \| h \|_{\beta}^2 &= \vert h(\infty)\vert^2+\int_0^{\infty}(h'(x))^2 \mathrm{e}^{\beta x}dx<\infty\
\end{align}
Such space was first motivated and introduced by Filipovic \cite{FilipoConsistencyproblemsforHeath}. Note that $h(\infty):=\lim_{x\to\infty}h(x)$ exists, whenever $h \in H_{\beta}$. It is called the \textit{long rate} of the forward curve $h$. The HJMM-equation on $H_{\beta}$ is given by
\begin{equation}
\begin{cases}\label{HJMM}
dX_t = \left(AX_t+F_{HJMM}(\sigma,\gamma)(X_t)\right) dt + \sigma(X_t)dW_t + \int_E\gamma(X_{t},\nu)\widetilde{N}(dt,d\nu),\\
X_0=h_0 \in L^2(\Omega, \mathcal{F}_0, \P; H_{\beta})
\end{cases}
\end{equation}
 where $(W_t)_{t \geq 0}$ is a $Q$-Wiener process and $\widetilde{N}(dt,d\nu)$
 is a compensated Poisson random measure on $E$ with compensator $dt \mu(d\nu)$
 as defined in Section 2 for $H := H_{\beta}$, and 
\begin{itemize}
 \item[(i)] $A$ is the infinitesimal generator of the shift semigroup $(S(t))_{t\in\mathbb{R}_+}$ on $H_{\beta}$, that is $S(t)h(x):=h(x+t)$ for all $t,x\geq 0$.
 \item[(ii)] $h\mapsto \sigma(h)$ is a $\mathcal{B}(H_{\beta})/\mathcal{B}(L_2^0)$-measurable mapping from  $H_{\beta}$ into $L_2^0(H_{\beta})$ and $(h,\nu)\mapsto\gamma(h,\nu)$ is $\mathcal{B}(H_{\beta})\otimes\mathcal{E}/\mathcal{B}(H_{\beta})$-measurable mapping from  $H_{\beta}\times E$ into $H_{\beta}$. 
 \item[(iii)] The drift is of the form
 \begin{align*}
F_{HJMM}(\sigma,\gamma)(h) = \sum_{j\in \mathbb{N}}\sigma^j(h) \Sigma^j(h)-\int_E\gamma(h,\nu)(\mathrm{e}^{\Gamma(h,\nu)}-1)\mu(d\nu),
\end{align*}
 with $\sigma^j(h) = \sqrt{\lambda_j}\sigma(h)e_j$, 
 \[
  \Sigma^j(h)(t) = \int_0^t\sigma^j(h)(s)ds \ \text{ and } \ \Gamma(h,\nu)(t) = -\int_0^t \gamma(h,\nu)(s)ds.
 \]
\end{itemize}
The special form of the drift stems from mathematical finance and is sufficient
for the absence of arbitrage opportunities. We denote the space of all forward rates with long rate equal to zero by
$$H_{\beta}^0=\lbrace h\in H_{\beta}: h(\infty)=0\rbrace.$$
For the construction of a unique mild solution to \eqref{HJMM} the following conditions have been introduced in \cite{Filipovi2010term}:
\begin{enumerate}
    \item[(B1)] $\sigma:H_{\beta}\to L_2^0(H_{\beta}^0)$, $\gamma:H_{\beta}\times E \to H_{\beta'}^0$ are Borel measurable for some $\beta' > \beta$. 
    \item[(B2)] There exists a function $\Phi:E\to \mathbb{R_+}$ such that $\Phi(\nu)\geq\vert \Gamma(h,\nu)(t)\vert$ for all $h\in H_{\beta}$, $\nu\in E$ and $t\geq 0$. 
    \item[(B3)] There is an $M\geq 0$ such that, for all $h\in H_{\beta}$, and some $\beta' > \beta$
\begin{align*}
\| \sigma(h)\|_{L_2^0(H_{\beta})} \leq M,\quad  \int_E \mathrm{e}^{\Phi(\nu)} \max\lbrace \|\gamma(h,\nu)\|_{\beta'}^2,\|\gamma(h,\nu)\|_{\beta'}^4\rbrace \mu(d\nu) \leq M.
\end{align*}
\item[(B4)] The function
$F_2:H_{\beta}\to H_{\beta}^0$ defined by 
 \[
  F_2(h)=-\int_E \gamma(h,\nu) \left(\mathrm{e}^{\Gamma(h,\nu)}-1 \right)\mu(d\nu)
 \]
 has the weak derivative given by 
 \[
  \frac{d}{dx}F_2(h)=\int_E \gamma(h,\nu)^2\left( \mathrm{e}^{\Gamma(h,\nu)} \right)\mu(d\nu)-\int_E \left(\frac{d}{dx}\gamma(h,\nu)\right)\left(\mathrm{e}^{\Gamma(h,\nu)}-1\right)\mu(d\nu).
 \]
\item[(B5)] There are constants $L_{\sigma}, L_{\gamma}>0$ such that, for all $h_1,h_2 \in H_{\beta}$, we have
 \begin{align*}
\| \sigma(h_1)-\sigma(h_2)\|_{L_2^0(H_{\beta})}^2 \leq L_{\sigma} \| h_1 - h_2 \|_{\beta}^2,
\\ \int_E \mathrm{e}^{\Phi(\nu)} \| \gamma(h_1,\nu)-\gamma(h_2,\nu)\|_{\beta'}^2 \mu(d\nu)\leq L_{\gamma} \| h_1 - h_2 \|_{\beta}^2.
\end{align*}
\end{enumerate} 
The following is the basic existence and uniqueness result for the Heath-Jarrow-Morton-Musiela equation \eqref{HJMM}.
\begin{Theorem}\cite{Filipovi2010term}\label{T: HJMUNIQUEANDEXIST}
 Suppose that conditions (B1) -- (B5) are satisfied.
 Then $F_{HJMM}: H_{\beta} \longrightarrow H_{\beta}^0$ and there exists a constant $L_{F}>0$ such that, for each $h_1,h_2\in H_{\beta}$,
 \begin{align}\label{eq: HJM drift lipschitz}
  \| F_{HJMM}(h_1) - F_{HJMM}(h_2)\|_{\beta}^2\leq L_F \| h_1 -h_2\|_{\beta}^2.
 \end{align}
 This constant can be choosen as 
\begin{equation}\label{Kestimate}
L_F = \frac{\max(L_{\sigma},L_{\gamma})\sqrt{M}}{\beta}\left(\sqrt{6M\sqrt{2}} +\sqrt{\frac{8}{\beta^3}   + \frac{16}{\beta} }+\sqrt{ \frac{16  (1+\frac{1}{\sqrt{\beta}})^2+48}{(\beta'-\beta)}} \right).
\end{equation}
Moreover, for each initial condition $h \in L^2(\Omega, \mathcal{F}_0, \P; H_{\beta})$ there is a unique adapted, c\'{a}dl\'{a}g mild solution $(r_t)_{t\geq 0}$ to \eqref{HJMM}.
\end{Theorem}
\begin{proof}
This result can be found essentially in \cite{Filipovi2010term}, where the bound on $L_F$ is an immediate result from its derivation.
\end{proof}
Using the space of all functions with zero long rate we obtain the decomposition
\[
 H_{\beta} = H_{\beta}^0 \oplus \R, \qquad h = (h - h(\infty)) + h(\infty),
\]
where $h(\infty) \in \R$ is identified with a constant function. Denote by 
\begin{align*}
 P_0h = h - h(\infty) \ \text{ and } \ P_1 h = h(\infty)
\end{align*}
the corresponding projections onto $H_{\beta}^0$ and $\R$, respectively.
Such a decomposition of $H_{\beta}$ was first used in \cite{Tehranchi2005} to study invariant measures for the HJMM-equation driven by a $Q$-Wiener process. An extension to the L\'evy driven HJMM-equation was then obtained in \cite{Rusinek10}. The next theorem shows that the results of Section 4 contain the HJMM-equation as a particular case.
\begin{Theorem}\label{T: Ergodicity for HJM}
Suppose that conditions (B1) -- (B5) are satisfied.
If 
\begin{align}\label{HJMcontraction condition}
 \beta > 2 \sqrt{L_F} + L_{\sigma} + L_{\gamma},
\end{align}
then for each initial distribution $\rho$ on $H_{\beta}$ with finite second moments there exists an invariant measure $\pi_{\rho}$ and it holds that
\begin{align}\label{ergodicity HJMM}
\mathrm{W}_2(\mathcal{P}_t^*\rho, \pi_{\rho}) \leq K\left( 1 + \int_{H_{\beta}} \| h \|_{H_{\beta}}^2 \rho(dh) + \int_{H_{\beta}} \| h \|_{H_{\beta}}^2 \pi_{\rho}(dh) \right) \mathrm{e}^{- \frac{\beta - 2\sqrt{L_F} - L_{\sigma} - L_{\gamma}}{2} t}
\end{align}
for some constant $K = K(\beta, \sigma, \gamma) > 0$.
Moreover, given $\rho, \widetilde{\rho}$
such that $\rho \circ P_1^{-1} = \widetilde{\rho} \circ P_1^{-1}$, then $\pi_{\rho} = \pi_{\widetilde{\rho}}$.
\end{Theorem}
\begin{proof}
 Observe that the assertion is an immediate consequence of Theorem \ref{theorem: existence limits} and Corollary \ref{T: Uniqueness on affine Subspaces}. Below we briefly verify the assumptions given in these statements. Condition (A1) follows from (B1), (B5), and \eqref{eq: HJM drift lipschitz}
 while the growth condition \eqref{abstractgrowthsforjumpcoeff} is satisfied by (B3) and the fact that $\| \cdot \|_{\beta} \leq \| \cdot \|_{\beta'}$ for $\beta < \beta'$. It is not difficult to see that
 \[
  \| S(t)h - P_1h \|_{\beta} \leq \mathrm{e}^{-\frac{\beta}{2}t} \| h - P_1h \|_{\beta}, \qquad t \geq 0
 \]
 and that $(S(t))_{t \geq 0}$ leaves $H_{\beta}^0$ as well as $\R \subset H_{\beta}$ invariant. 
 Hence Remark \ref{example: ergodic semigroup} yields that
 \[
  \langle Ah,h \rangle \leq -\frac{\beta}{2}\|h\|_{\beta}^2 + \frac{\beta}{2} \|P_1h\|_{\beta}^2, \qquad h \in D(A).
 \]
 It follows from the considerations in Section 2 (see \eqref{eq:03}) that (GDC) is satisfied for $\alpha = \frac{\beta}{2} - \sqrt{L_F}$.
 Consequently, $\e = \beta - 2\sqrt{L_F} - L_{\sigma} - L_{\gamma}$ and \eqref{definition epsilon} holds due to \eqref{HJMcontraction condition}.
 Since the coefficients map into $H_{\beta}^0$ and $S(t) P_1 h =h(\infty)=P_1 h$, conditions (A2), (A3) and \eqref{eq:04} are trivially satisfied. The particular form of the estimate \eqref{ergodicity HJMM} follows from the proof of Theorem \ref{theorem: existence limits}.
\end{proof}
Comparing our result with \cite{Tehranchi2005, Rusinek10},
we allow for a more general jump noise and prove convergence in the stronger Wasserstein distance with an exponential rate. 
Moreover, assuming that the volatilities map constant functions onto zero,
i.e. 
\begin{align}\label{HJMVolatilities Zero in Constants}
\sigma(c)\equiv 0,\quad
\gamma(c,\nu)\equiv 0, \qquad \forall c \in \R \subset H_{\beta}, \ \nu \in E
\end{align}
shows that $F(c) \equiv 0$ and hence also \eqref{Coefficients vanish in H1} is satisfied.
Hence we may apply Theorem \ref{thm: HJMcase abstract} to characterize these invariant measures more explicitly.
\begin{Corollary}\label{C: Identifiable Limits for HJM}
Suppose that conditions (B1) -- (B5) are satisfied, that \eqref{HJMcontraction condition} and \eqref{HJMVolatilities Zero in Constants} hold.
Then 
\[
 \E\left[ \| X_t^h - h(\infty)\|_{\beta}^2 \right]
 \leq  \E\left[ \| h - h(\infty)\|_{\beta}^2 \right]\mathrm{e}^{- \left(\beta - 2 \sqrt{L_F} - L_{\sigma} - L_{\gamma}\right)t}
\]
for each $h\in L^2(\Omega,\mathbb{F}_0,\mathbb{P},H)$.
\end{Corollary}
We close this section by applying our results for the particular example of coefficients as introduced in \cite{Rusinek10}.
\begin{Example}
 Take 
 \[
  \sigma^1(h)(x) := \int_x^{\infty}\min\left( \mathrm{e}^{- \beta y},\ |h'(y)| \right)dy
 \] 
and $\sigma^j\equiv 0$ for $j\geq 2$. Then
\[
 \| \sigma(h) \|_{L_2^0}^2 = \| \sigma^1(h)\|_{\beta}^2\leq\int_0^\infty (\mathrm{e}^{-2\beta x}) \mathrm{e}^{\beta x} dx=\frac{1}{\beta} =:M
\]
and since $\min(a,b_1)-\min(a,b_2)\leq |b_1-b_2|$ for $a,b_1,b_2\in\mathbb{R}_+$, we also have
\begin{align*}
\| \sigma(h_1)-\sigma(h_2) \|_{L_2^0}^2& = \| \sigma^1(h)-\sigma^1(h_2)\|_{\beta}^2\\
&=\int_0^\infty (\min(\mathrm{e}^{-\beta x},|h_1'(x)|)-\min(\mathrm{e}^{-\beta x},|h_2'(x)|))^2 \mathrm{e}^{\beta x} dx\\
&\leq \int_0^\infty (h_1'(x)-h_2'(x))^2 \mathrm{e}^{\beta x} dx\\
&\leq \| h_1-h_2\|_{\beta}^2.
\end{align*}
This shows that the Lipschitz condition for $\sigma$ holds, in particular, for $L_{\sigma} = 1$.  
Consequently, by taking $\gamma \equiv 0$, the conditions (B1) -- (B5) are satisfied with $L_{\sigma}=1$ and $L_{\gamma} = 0$ and $M=\frac{1}{\beta}$ for the Lipschitz and growth constants.
By \eqref{Kestimate} we get
\[
 L_F = \frac{1}{\sqrt{\beta^3}}\left(\sqrt{\frac{6\sqrt{2}}{\beta}} +\sqrt{\frac{8}{\beta^3}   + \frac{16}{\beta} } + \sqrt{ \frac{16  (1+\frac{1}{\sqrt{\beta}})^2+48}{(\beta'-\beta)}} \right),
\]
for all $\beta' > \beta$. Choosing $\beta \geq 3$ and $\beta' > \beta$ large enough such that
$L_F < 1$, we find that
\begin{align*}
    2 \sqrt{L_F} + L_{\sigma} + L_{\gamma}
    &< 3 = \beta,
\end{align*}
i.e. \eqref{HJMcontraction condition} is satisfied. 
It is clear that $\sigma(c) \equiv 0$ for each constant function $c$.
Hence Corollary \ref{C: Identifiable Limits for HJM} is applicable.
\end{Example}

\appendix

\section{It\^{o} formula}

Below we recall an It\^{o} formula for Hilbert space valued semimartingales of the form
\[
X(t)=X(0)+\int_0^t a(s) ds + \int_0^t \sigma(s) dW_s+\int_0^{t} \int_E \gamma(s,\nu)\widetilde{N}(ds,d\nu),
\]
where $a$ and $\sigma$ are as before and $(\gamma(t,\nu))_{t\geq 0}$ is a predictable, $H$-valued stochastic process for each $\nu \in E$ such that
\[
 \mathbb{E}\left[\int_0^t \int_E \| \gamma(s,\nu)\|_H^2 \mu(d\nu)ds \right ] <\infty
\]
and 
\[
 \mathbb{E} \left[\int_0^t  \| \sigma(s)\|_{L_2^0}^2 ds \right ]<\infty.
\]
For this purpose we first introduce the class of quasi-sublinear functions.
\begin{Definition}[Sublinear Functions]
A continuous, non-decreasing function $h:\mathbb{R}_+\to\mathbb{R}_+$ is called quasi-sublinear, if there exists a constant $C>0$ such that
\begin{align*}
h(x+y)\leq C (h(x)+h(y))\\
h(xy)\leq C (h(x)h(y))
\end{align*}
for all $x,y\geq 0$.
\end{Definition}
The following It\^{o}-Formula is a combination of \cite{gawarecki2010stochastic} and \cite{Mandrekar2013}. 
\begin{Theorem}[Generalized It\^{o}-Formula]\label{GenIto}
Let $F \in C^2(\mathbb{R}_+\times H,\mathbb{R})$
and suppose there exist quasi-sublinear functions $h_1,h_2:\mathbb{R}_+\to\mathbb{R}_+$ such that for all $t \geq 0$ and $x \in H$
\begin{align*}
 \|F_x(t,x)\|_H \leq h_1(\| x \|_H), \qquad 
 \|F_{xx}(t,x)\|_{L(H,L(H,\mathbb{R}))} \leq h_2(\| x \|_H)
\end{align*}
and 
\begin{align*}
 \int_0^t\int_E \| \gamma(s,\nu)\|_H^2 \mu(d\nu)ds + \int_0^t\int_E h_1(\| \gamma(s,\nu)\|_H)^2 \| \gamma(s,\nu)\|_H^2 \mu(d\nu)ds
 \\ +\int_0^t\int_E h_2(\| \gamma(s,\nu)\|_H) \| \gamma(s,\nu)\|_H^2 \mu(d\nu)ds < \infty
\end{align*}
Then $\mathbb{P}$-almost surely for each $t \geq 0$:
\begin{align*}
\int_0^t \| F_t(s,X(s))\|_H ds +
\int_0^t\int_E \vert F(s,X(s)+\gamma(s,\nu))-F(s,X(s))\vert^2 \mu(d\nu)ds
\\ + \int_0^t\int_E \vert F(s,X(s)+\gamma(s,\nu))-F(s,X(s))-\langle F_x(s,X(s)),\gamma(s,\nu)\rangle_H)\vert \mu(d\nu)ds<\infty.
\end{align*}
Moreover, the generalized It\^{o}-formula holds $\mathbb{P}$-almost surely for each $t \geq 0$ and 
\begin{align*}
 F(t,X(t)) &= F(0,X(0)) + \int_0^t \mathcal{L}F(s,X(s)) ds
 \\ &+ \int_0^t \langle  F_x(s,X(s)), \sigma(s) dW_s \rangle_H
 \\ &\ \ \ + \int_0^{t+}\int_E \left\{ F(s,X(s-)+\gamma(s,\nu)) - F(s,X(s-)) \right\}\widetilde{N}(ds, d\nu)
\end{align*}
where $\mathcal{L}F(x,X(s))$ is given by
\begin{align*}
   &\ \mathcal{L}F(s,X(s)) 
   \\ &= \int_0^t \left\{ F_t(s,X(s)) + \langle F_x(s,X(s)),a(s)\rangle_H \right\}ds
   \\ &\ \ \ + \frac{1}{2}\int_0^t  \mathrm{tr}\left[F_{xx}(s,X(s))\sigma(s)Q\sigma(s)^{*}\right] ds
   \\ &\ \ \ + \int_0^t \int_E \left\{ F(s,X(s)+\gamma(s,\nu))-F(s,X(s))-\langle F_x(s,X(s)),\gamma(s,\nu)\rangle_H \right\} \mu(d\nu)ds
\end{align*}
\end{Theorem}

\subsection*{Acknowledgements}
Dennis Schroers was funded within the project STORM: Stochastics for Time-Space Risk Models, from the Research Council of Norway (RCN). Project number: 274410.

\bibliographystyle{amsplain}
\addcontentsline{toc}{section}{\refname}\bibliography{references}

\providecommand{\bysame}{\leavevmode\hbox to3em{\hrulefill}\thinspace}
\providecommand{\MR}{\relax\ifhmode\unskip\space\fi MR }
\providecommand{\MRhref}[2]{%
  \href{http://www.ams.org/mathscinet-getitem?mr=#1}{#2}
}
\providecommand{\href}[2]{#2}
\begin{thebibliography}{10}

\bibitem{ALBEVERIO2009835}
S.~Albeverio, V~Mandrekar, and B.~R\"{u}diger, \emph{Existence of mild
  solutions for stochastic differential equations and semilinear equations with
  non-gaussian l\'evy noise}, Stochastic Processes and their Applications
  \textbf{119} (2009), no.~3, 835 -- 863.

\bibitem{Albeverio2017}
Sergio Albeverio, Leszek Gawarecki, Vidyadhar Mandrekar, Barbara R\"{u}diger,
  and Barun Sarkar, \emph{{{It{\^{o}} Formula for Mild Solutions of {SPDEs}
  with Gaussian and non-Gaussian noise and Applications to Stability
  Properties}}}, Random Operators and Stochastic Equations \textbf{25} (2017),
  no.~2.

\bibitem{A15}
David Applebaum, \emph{Infinite dimensional {O}rnstein-{U}hlenbeck processes
  driven by {L}\'{e}vy processes}, Probab. Surv. \textbf{12} (2015), 33--54.
  \MR{3385977}

\bibitem{MR3474409}
Viorel Barbu, Michael R\"{o}ckner, and Deng Zhang, \emph{Stochastic nonlinear
  {S}chr\"{o}dinger equations}, Nonlinear Anal. \textbf{136} (2016), 168--194.
  \MR{3474409}

\bibitem{Btkai2005SemigroupsFD}
Andr\'{a}s B\'{a}tkai and Susanna Piazzera, \emph{Semigroups for delay
  equations}, Research Notes in Mathematics, vol.~10, A K Peters, Ltd.,
  Wellesley, MA, 2005. \MR{2181405}

\bibitem{BDK19}
Fred~Espen Benth, Nils Detering, and Paul Kruehner, \emph{Stochastic {V}olterra
  integral equations and a class of first order stochastic partial differential
  equations}, arXiv:1903.05045 (2019).

\bibitem{MR3228971}
Joris Bierkens and Onno van Gaans, \emph{Dissipativity of the delay semigroup},
  J. Differential Equations \textbf{257} (2014), no.~7, 2418--2429.
  \MR{3228971}

\bibitem{MR3178490}
Oleg Butkovsky, \emph{Subgeometric rates of convergence of {M}arkov processes
  in the {W}asserstein metric}, Ann.\ Appl.\ Probab. \textbf{24} (2014), no.~2,
  526--552. \MR{3178490}

\bibitem{MR2005200}
Giuseppe Da~Prato and Arnaud Debussche, \emph{Ergodicity for the 3{D}
  stochastic {N}avier-{S}tokes equations}, J. Math. Pures Appl. (9) \textbf{82}
  (2003), no.~8, 877--947. \MR{2005200}

\bibitem{da2014stochastic}
Giuseppe Da~Prato and Jerzy Zabczyk, \emph{{{Stochastic Equations in Infinite
  Dimensions}}}, Encyclopedia of Mathematics and its Applications, Cambridge
  University Press, 2014.

\bibitem{Filipovi2010term}
Stefan~Tappe Damir~Filipovi{\'{c}} and Josef Teichmann, \emph{{Term Structure
  Models Driven by Wiener Processes and Poisson Measures: Existence and
  Positivity}}, {SIAM} Journal on Financial Mathematics \textbf{1} (2010),
  no.~1, 523--554.

\bibitem{MR1921744}
Donald~A. Dawson, Alison~M. Etheridge, Klaus Fleischmann, Leonid Mytnik,
  Edwin~A. Perkins, and Jie Xiong, \emph{Mutually catalytic branching in the
  plane: infinite measure states}, Electron. J. Probab. \textbf{7} (2002), No.
  15, 61. \MR{1921744}

\bibitem{MR1954077}
A.~de~Bouard and A.~Debussche, \emph{The stochastic nonlinear {S}chr\"{o}dinger
  equation in {$H^1$}}, Stochastic Anal. Appl. \textbf{21} (2003), no.~1,
  97--126. \MR{1954077}

\bibitem{FilipoConsistencyproblemsforHeath}
Damir Filipovi{\'c}, \emph{{{Consistency {P}roblems for {HJM} {I}nterest {R}ate
  {M}odels}}}, Springer, Berlin ; Heidelberg ; New York ; Barcelona ; Hong Kong
  ; London ; Milan ; Paris ; Singapore ; Tokyo, 2001.

\bibitem{filipovic2009term}
Damir Filipovi{\'{c}}, \emph{{{Term-Structure Models: A Graduate Course}}},
  Springer Finance, Springer Berlin Heidelberg, 2009.

\bibitem{Filipovi2010jump}
Damir Filipovi{\'{c}}, Stefan Tappe, and Josef Teichmann,
  \emph{{{Jump-Diffusions in {H}ilbert Spaces: Existence, Stability and
  Numerics}}}, Stochastics \textbf{82} (2010), no.~5, 475--520.

\bibitem{MR1300150}
Franco Flandoli, \emph{Dissipativity and invariant measures for stochastic
  {N}avier-{S}tokes equations}, NoDEA Nonlinear Differential Equations Appl.
  \textbf{1} (1994), no.~4, 403--423. \MR{1300150}

\bibitem{FJR19}
Martin {Friesen}, Peng {Jin}, and Barbara {R{\"u}diger}, \emph{{Stochastic
  equation and exponential ergodicity in Wasserstein distances for affine
  processes}}, Ann. Appl. Prob. (2020), (to appear).

\bibitem{gawarecki2010stochastic}
Leszek Gawarecki and Vidyadhar Mandrekar, \emph{{{Stochastic Differential
  Equations in Infinite Dimensions: with Applications to Stochastic Partial
  Differential Equations}}}, Probability and Its Applications, Springer Berlin
  Heidelberg, 2010.

\bibitem{MR2773030}
M.~Hairer, J.~C. Mattingly, and M.~Scheutzow, \emph{Asymptotic coupling and a
  general form of {H}arris' theorem with applications to stochastic delay
  equations}, Probab.\ Theory Related Fields \textbf{149} (2011), no.~1-2,
  223--259. \MR{2773030}

\bibitem{MR2259251}
Martin Hairer and Jonathan~C. Mattingly, \emph{Ergodicity of the 2{D}
  {N}avier-{S}tokes equations with degenerate stochastic forcing}, Ann. of
  Math. (2) \textbf{164} (2006), no.~3, 993--1032. \MR{2259251}

\bibitem{MR2857021}
\bysame, \emph{Yet another look at {H}arris' ergodic theorem for {M}arkov
  chains}, Seminar on {S}tochastic {A}nalysis, {R}andom {F}ields and
  {A}pplications {VI}, Progr.\ Probab., vol.~63, Birkh\"{a}user/Springer Basel
  AG, Basel, 2011, pp.~109--117. \MR{2857021}

\bibitem{HJMM92}
David Heath, Robert Jarrow, and Andrew Morton, \emph{{{Bond Pricing and the
  Term Structure of Interest Rates: A New Methodology for Contingent Claims
  Valuation}}}, Econometrica \textbf{60} (1992), 77--105.

\bibitem{MR3800835}
Alexei Kulik and Michael Scheutzow, \emph{Generalized couplings and convergence
  of transition probabilities}, Probab.\ Theory Related Fields \textbf{171}
  (2018), no.~1-2, 333--376. \MR{3800835}

\bibitem{MR2336594}
Thomas~G. Kurtz, \emph{The {Y}amada-{W}atanabe-{E}ngelbert theorem for general
  stochastic equations and inequalities}, Electron. J. Probab. \textbf{12}
  (2007), 951--965. \MR{2336594}

\bibitem{9783319128535}
Vidyadhar Mandrekar and Barbara R\"{u}diger, \emph{{{Stochastic Integration in
  Banach Spaces: Theory and Applications (Probability Theory and Stochastic
  Modelling Book 73)}}}, Springer, 2014.

\bibitem{Mandrekar2013}
Vidyadhar Mandrekar, Barbara R\"{u}diger, and Stefan Tappe, \emph{{{It{\^{o}}'s
  Formula for Banach-space-valued Jump Processes Driven by Poisson Random
  Measures}}}, Seminar on Stochastic Analysis, Random Fields and Applications
  {VII}, Springer Basel, 2013, pp.~171--186.

\bibitem{MR1653845}
Leonid Mytnik, \emph{Uniqueness for a mutually catalytic branching model},
  Probab. Theory Related Fields \textbf{112} (1998), no.~2, 245--253.
  \MR{1653845}

\bibitem{MR710486}
A.~Pazy, \emph{Semigroups of linear operators and applications to partial
  differential equations}, Applied Mathematical Sciences, vol.~44,
  Springer-Verlag, New York, 1983. \MR{710486}

\bibitem{peszat2007stochastic}
Szymon Peszat and Jerzy Zabczyk, \emph{{{Stochastic Partial Differential
  Equations with L{\'e}vy Noise: An Evolution Equation Approach}}},
  Encyclopedia of Mathematics and its Applications, Cambridge University Press,
  2007.

\bibitem{MR2860444}
Enrico Priola, Armen Shirikyan, Lihu Xu, and Jerzy Zabczyk, \emph{Exponential
  ergodicity and regularity for equations with {L}\'{e}vy noise}, Stochastic
  Process. Appl. \textbf{122} (2012), no.~1, 106--133. \MR{2860444}

\bibitem{MR2836539}
Enrico Priola, Lihu Xu, and Jerzy Zabczyk, \emph{Exponential mixing for some
  {SPDE}s with {L}\'{e}vy noise}, Stoch. Dyn. \textbf{11} (2011), no.~2-3,
  521--534. \MR{2836539}

\bibitem{MR2985090}
Enrico Priola and Jerzy Zabczyk, \emph{On linear evolution equations for a
  class of cylindrical {L}\'{e}vy noises}, Stochastic partial differential
  equations and applications, Quad. Mat., vol.~25, Dept. Math., Seconda Univ.
  Napoli, Caserta, 2010, pp.~223--242. \MR{2985090}

\bibitem{Rusinek10}
Anna Rusinek, \emph{{{Mean Reversion for HJMM Forward Rate Models}}}, Advances
  in Applied Probability \textbf{42} (2010), no.~2, 371–391.

\bibitem{MR2269220}
S.~S. Sritharan and P.~Sundar, \emph{Large deviations for the two-dimensional
  {N}avier-{S}tokes equations with multiplicative noise}, Stochastic Process.
  Appl. \textbf{116} (2006), no.~11, 1636--1659. \MR{2269220}

\bibitem{Tehranchi2005}
Michael Tehranchi, \emph{{{A Note on Invariant Measures for {HJM} Models}}},
  Finance and Stochastics \textbf{9} (2005), no.~3, 389--398.

\bibitem{Gaans2005InvariantMF}
Onno van Gaans, \emph{Invariant measures for stochastic evolution equations
  with {H}ilbert space valued {L}{\'e}vy noise}, 2005.

\bibitem{MR2459454}
C\'{e}dric Villani, \emph{Optimal transport}, Grundlehren der Mathematischen
  Wissenschaften [Fundamental Principles of Mathematical Sciences], vol. 338,
  Springer-Verlag, Berlin, 2009, Old and new. \MR{2459454}

\end{thebibliography}

\end{document}